\documentclass[reqno,dvipsnames]{amsart}
\pdfoutput=1

\usepackage{amssymb}
\usepackage{amsmath}
\usepackage{mathabx}
\usepackage{hyperref}
\usepackage{graphicx}
\usepackage{mathrsfs}

\usepackage{epigraph}   

\usepackage{amssymb, amsmath}
\usepackage{mathrsfs}
\usepackage{amscd} 
\usepackage{amsmath}
\usepackage[english]{babel}
\usepackage{microtype}               
\usepackage[T1]{fontenc}
\usepackage{lmodern}

\usepackage{hyperref}
\hypersetup{
	colorlinks   = true,          
	urlcolor     = blue,          
	linkcolor    = blue,          
	citecolor   = red             
}

\usepackage{enumitem}
\setenumerate[1]{label=(\roman*)}    

\usepackage[normalem]{ulem}
\usepackage{graphicx}
\usepackage[all,cmtip]{xy}
\usepackage{tikz}
\usetikzlibrary{arrows,decorations,shapes,shadows}
\usetikzlibrary{calc,intersections}
\usepackage{verbatim}

\usepackage{epigraph}   

\newbox\mybox
\def\overtag#1#2#3{\setbox\mybox\hbox{$#1$}\hbox to
  0pt{\vbox to 0pt{\vglue-#3\vglue-\ht\mybox\hbox to \wd\mybox
      {\hss$\ss#2$\hss}\vss}\hss}\box\mybox}
\def\undertag#1#2#3{\setbox\mybox\hbox{$#1$}\hbox to 0pt{\vbox to
    0pt{\vglue#3\vglue\ht\mybox\hbox to \wd\mybox
      {\hss$\ss#2$\hss}\vss}\hss}\box\mybox}
\def\lefttag#1#2#3{\hbox to 0pt{\vbox to 0pt{\vglue -6pt\hbox to
      0pt{\hss$\ss#2$\hskip#3}\vss}}#1}
\def\righttag#1#2#3{\hbox to 0pt{\vbox to 0pt{\vglue -6pt\hbox to
      0pt{\hskip#3$\ss#2$\hss}\vss}}#1}
\let\ss\scriptstyle
\def\splicediag#1#2{\xymatrix@R=#1pt@C=#2pt@M=0pt@W=0pt@H=0pt}
\def\Dot{\lower.2pc\hbox to 2pt{\hss$\bullet$\hss}}
\def\Circ{\lower.2pc\hbox to 2pt{\hss$\circ$\hss}}
\def\Vdots{\raise5pt\hbox{$\vdots$}}
\newcommand\lineto{\ar@{-}}
\newcommand\dashto{\ar@{--}}
\newcommand\dotto{\ar@{.}}

\let\cal\mathcal
\renewcommand{\setminus}{\smallsetminus}

\newcommand\Q{{\mathbb Q}}
\newcommand\Pro{{\mathbb P}}
\newcommand\R{{\mathbb R}}
\newcommand\C{{\mathbb C}}
\newcommand\Z{{\mathbb Z}}
\newcommand\N{{\mathbb N}}

\DeclareMathOperator{\NL}{NL}

\DeclareMathOperator{\val}{val}
\DeclareMathOperator{\ord}{ord}

\DeclareMathOperator{\mult}{mult}
\DeclareMathOperator{\inn}{inn}
\DeclareMathOperator{\out}{out}
\DeclareMathOperator{\Gr}{Gr}

\newtheorem{theorem}{Theorem}[section]
\newtheorem{proposition}[theorem]{Proposition}
\newtheorem*{theorem*}{Theorem}
\newtheorem{corollary}[theorem]{Corollary}
\newtheorem{lemma}[theorem]{Lemma}

\theoremstyle{definition}

\newtheorem*{amalgamation*}{Amalgamation}
\newtheorem{example}[theorem]{Example}
\newtheorem{remark}[theorem]{Remark}
\newtheorem{remarks}[theorem]{Remarks}
\newtheorem*{remark*}{Remark}

\theoremstyle{plain}
\newcounter{algorithm}

\newtheorem{algo}[algorithm]{Algorithm}

\renewcommand{\int}{\operatorname{int}}
\newcommand{\lcm}{\operatorname{lcm}}

\keywords{Complex Surface Singularities, Lipschitz Geometry, Lipschitz Normal Embeddings, Polar Varieties, Discriminant Varieties, Valuation Spaces}

\subjclass{Primary 32S25; Secondary 13A18, 14B05}

\begin{document}

\title{On Lipschitz Normally Embedded complex surface germs}

\author{Andr\'e Belotto da Silva}
\address{IMJ-PRG, CNRS 7586, Universit\'e de Paris, Institut de Math\'ematiques de Jussieu Paris Rive Gauche, France}
\email{\href{mailto:andre.belotto@imj-prg.fr}{andre.belotto@imj-prg.fr}}
\urladdr{\url{https://andrebelotto.com}}

\author{Lorenzo Fantini}
\address{Centre de Mathématiques Laurent Schwartz, Ecole Polytechnique and CNRS, Institut Polytechnique de Paris,}
\email{\href{mailto:lorenzo.fantini@polytechnique.edu}{lorenzo.fantini@polytechnique.edu}}
\urladdr{\url{https://lorenzofantini.eu/}}

\author{Anne Pichon}
\address{Aix-Marseille Universit\'e, CNRS, Centrale Marseille, I2M, Marseille, France}
\email{\href{mailto:anne.pichon@univ-amu.fr}{anne.pichon@univ-amu.fr}}
\urladdr{\url{http://iml.univ-mrs.fr/~pichon/}}

\begin{abstract}
	We undertake a systematic study of Lipschitz Normally Embedded normal complex surface germs.
	We prove in particular that the topological type of such a germ determines the combinatorics of its minimal resolution which factors through the blowup of its maximal ideal and through its Nash transform, as well as the polar curve and the discriminant curve of a generic plane projection, thus generalizing results of Spivakovsky and Bondil that were known for minimal surface singularities.
	In an appendix, we give a new example of a Lipschitz Normally Embedded surface singularity.
\end{abstract}

\maketitle

 
 \section{Introduction}
 

A germ of a real or complex analytic space $(X,0)$ embedded in $(\R^n,0)$ or in $(\C^n,0)$ is equipped with two natural metrics: its \emph{outer metric} $d_{\out}$, induced by the standard metric of the ambient space, and its \emph{inner metric} $d_{\inn}$, which is the associated arc-length metric on the germ. 
These two metrics are usually studied up to bilipschitz local homeomorphisms, since they then give rise to tame classifications of singular sets, as was proven in various geometric contexts by Pham and Teissier~\cite{PhamTeissier1969}, Mostowski~\cite{Mostowski1985, Mostowski1988}, Parusi\'nski~\cite{Parusinski1988, Parusinski1994}, and Valette \cite{Valette2005}.

The  germ  $(X,0)$ is said to be \emph{Lipschitz Normally Embedded} (\emph{LNE} for short) if the identity map of $(X,0)$ is a bilipschitz homeomorphism between the inner and the outer metric, that is if there exist a neighborhood $U$ of $0$ in $X$ and a constant $K\geq 1$ such that
\[
d_{\inn}(x,y) \leq K d_{\out}(x,y)
\]
for all $x$ and $y$ in $U$.
Since the inner and the outer geometries of $(X,0)$ are invariant under bi-Lipschitz homeomorphisms (see \cite[Proposition~7.2.13]{Pichon2020}), this property only depends on the analytic type of $(X,0)$, and not on the choice of an embedding in some smooth ambient space $(\R^n,0)$ or $(\C^n,0)$.

The study of Lipschitz Normally Embedded singularities is a very active research area with many recent results, for example by Birbrair, Bobadilla, Fernandes, Heinze, Kerner, Mendes, Misev, Neumann, Pedersen, Pereira, Pichon, Ruas, and Sampaio (see \cite{BirbrairMendes2018, BobadillaHeinzePereiraSampaio2019, FernandesSampaio2019, KernerPedersenRuas2018, NeumannPedersenPichon2020a, NeumannPedersenPichon2020b}), but despite the current progress it is still in its infancy. 

While an irreducible complex curve germ $(X,0)$ is LNE if and only if it is smooth (see \cite{PhamTeissier1969, Fernandes2003, NeumannPichon2007}), the situation is far richer already for complex surface germs.
Lipschitz Normally Embedded germs are fairly common in this context, including in particular all minimal surface singularities (as proven in \cite{NeumannPedersenPichon2020b} exploiting a characterization obtained in \cite{NeumannPedersenPichon2020a}), and the superisolated surface singularities with LNE tangent cone (see \cite{MisevPichon2018}).
In this paper we prove several properties of a general complex LNE normal surface, describing in particular its generic polar curves and the discriminant curves of its generic plane projections.
We also give a new example of a LNE normal singularity which is neither minimal nor superisolated, showing that the  class of LNE normal surface singularities contains more elements than the ones already discovered (see Appendix~\ref{appendix:example}).

Among LNE surface singularities, the most widely studied are minimal singularities, which have been introduced in greater generality in \cite{Kollar1985}.
In dimension two they are the rational surface singularities with reduced fundamental cycle, and they have the remarkable property that the topological type of $(X,0)$ determines the following data, which is a priori of analytic nature:
\begin{enumerate}[label=(\arabic*)]
	\item \label{list_intro:property_polar} The  dual graph of the minimal  good resolution of $(X,0)$ which factors through the blowup of the maximal ideal and through the Nash transform, decorated by two families of arrows corresponding to the strict transform of a generic hyperplane section and to the strict transform of the polar curve of a generic plane projection.
	\item \label{list_intro:property_discriminant} The topological type of the discriminant curve of a generic projection.
	Moreover, this data can be computed explicitly from the dual graph of the minimal good resolution of $(X,0)$.
\end{enumerate}

The first property is a deep result of Spivakovsky  \cite[III, Theorem 5.4]{Spivakovsky1990}, the second one was later proven by Bondil \cite[Theorem 4.1]{Bondil2003}, \cite[Proposition 5.4]{Bondil2016}.

Observe that by good resolution of $(X,0)$ we mean a proper bimeromorphic morphism $\pi\colon X_\pi\to X$ from a smooth surface $X_\pi$ to $X$ which is an isomorphism outside of a simple normal crossing divisor $E=\pi^{-1}(0)$, and the vertices of the dual graph $\Gamma_\pi$ of $E$ carry as weights the genera and self-intersections of the corresponding irreducible components of $E$.
The fact that the topological type of a surface germ determines the dual graph of its minimal resolution is a classical result of Neumann \cite{Neumann1981}.

The two main results of the present paper extend the theorems of Spivakovsky and Bondil to all LNE surface singularities.
Furthermore, we strengthen Spivakovsky's result by showing that another important datum is an invariant of the topological type of $(X,0)$, namely the {inner rates} of $(X,0)$, an infinite family of rational numbers which measures the local metric structure of the germ $(X,0)$ with respect to its inner metric.
If $E_v$ is a component of the exceptional divisor of a good resolution of $(X,0)$, then its inner rate $q_v$, introduced in \cite{BirbrairNeumannPichon2014} and further studied in \cite{BelottodaSilvaFantiniPichon2019}, measures the shrinking rate of the piece of the link of $(X,0)$ that corresponds to $E_v$ (see \cite[\S 1, \S3]{BelottodaSilvaFantiniPichon2019}).
These results show the crucial role played by generic projections and polar varieties, notions introduced and studied by Teissier~\cite{Teissier1982}, in the understanding of LNE singularities.
$$\diamond$$

In order to give a precise statement of our results we need to introduce some additional notation.
Let $\pi\colon X_\pi \to X$ be a good resolution of $(X,0)$ and denote by $V(\Gamma_\pi)$ the set of vertices of the dual graph $\Gamma_\pi$ of $\pi$, so that every element $v$ of $V(\Gamma_\pi)$ corresponds to a component $E_v $ of the exceptional divisor $E=\pi^{-1}(0)$ of $\pi$. 
We denote by $Z_{\max}(X,0)=\sum_{v\in V(\Gamma_\pi)}m_vE_v$ the maximal ideal divisor of $(X,0)$, that is the divisor of $X_\pi$ supported on $E$ and whose coefficient $m_v$, called \emph{multiplicity of $v$}, is the multiplicity along the component $E_v$ of the pullback via $\pi$ of a generic linear form $h\colon(X,0)\to(\C,0)$ on $(X,0)$.
While in general the divisor $Z_{\max}(X,0)$ depends on the analytic type of $(X,0)$, there is another divisor supported on $E$, namely the fundamental cycle $Z_{\min}$ of $\Gamma_\pi$, defined as the unique minimal nonzero element of the Lipman cone of $\Gamma_\pi$ (see Section~\ref{sec:Zmin=Zmax} for the relevant definitions), which only depends on the graph $\Gamma_\pi$. 
Finally, we denote by $Z_{\Gamma_{\pi}}$ the canonical cycle of $\Gamma_\pi$, that is the divisor supported on $E$ determined by $Z_{\Gamma_{\pi}} \cdot E_{v} = - E_{v}^2 + 2g(E_v) - 2$ for every vertex $v$ of $\Gamma_\pi$.

For each vertex $v$ of $\Gamma_\pi$, set  $l_v=-Z_{\max}(X,0)\cdot E_v$, that is,  $l_v$  is  the intersection multiplicity of $E_v$ with the strict transform of a generic hyperplane section $h^{-1}(0)$ of $(X,0)$ via $\pi$. 
We call \emph{$\cal L$-vector} of $(X,0)$ the vector $L_\pi=(l_v)_{v\in V(\Gamma_\pi)}\in\Z_{\geq0}^{V(\Gamma_\pi)}$.  
Recall that the blowup $\mathrm{Bl}_0X$ of the maximal ideal of $(X,0)$ is the minimal resolution of the base points of the family of generic hyperplane sections of $(X,0)$.
Whenever $\pi\colon X_\pi \to X$ factors through $\mathrm{Bl}_0X$, the strict transform of such a generic hyperplane section via $\pi$ consists of a disjoint union of smooth curves that intersect transversely $E$ at smooth points of $E$ (see Appendix~\ref{appendix:A}), and $l_v$ is the number of such curves passing through the component $E_v$; we then call \emph{$\cal L$-node} of $\Gamma_\pi$ (or simply of $(X,0)$) any vertex $v$ such that $l_v>0$.
Similarly, we denote by $p_v$ the intersection multiplicity of the strict transform of the polar curve of a generic plane projection $\ell\colon(X,0)\to(\C^2,0)$ with  $E_v$ and  we call \emph{$\cal P$-vector} of $(X,0)$ the vector $P_\pi=(p_v)_{v\in V(\Gamma_\pi)}\in\Z_{\geq0}^{V(\Gamma_\pi)}$.  
The Nash transform $\nu$ of $(X,0)$ is the minimal resolution of the base points of the family of  generic polar curves of $(X,0)$ (see \cite[Section III, Theorem 1.2]{Spivakovsky1990}). 
Whenever $\pi\colon X_\pi \to X$ factors through $\nu$ then such a strict transform consists of smooth curves intersecting $E$ transversely at smooth points, and $p_v$ equals the number of such curves through $E_v$ 
(see again Appendix~\ref{appendix:A}).

We then call \emph{$\cal P$-node} of $\Gamma_\pi$ (or simply of $(X,0)$) any vertex $v$ such that $p_v>0$.
Finally, whenever $\pi\colon X_\pi \to X$ factors through the blowup of the maximal ideal, we define a natural distance $d$ on $\Gamma_\pi$ by declaring the length of an edge $e$ between two vertices $v$ and $v'$ of $\Gamma_\pi$ to be $1/\lcm(m_v,m_{v'})$.

We can now state our first main theorem, which generalizes Spivakovsky's result \cite[III, Theorem 5.4]{Spivakovsky1990} to all LNE normal surface germs.

\begin{theorem} \label{thm:main}
	Let $(X,0)$ be a LNE normal surface germ, let $\pi\colon X_\pi \to X$ be the minimal good resolution of $(X,0)$, and let $\Gamma_\pi$ be the dual graph of $\pi$. 
	Then the following properties hold.	 
	\begin{enumerate}
		\item \label{thm:main_factorization_blowup} 
		The resolution $\pi$ factors through the blowup of the maximal ideal of $(X,0)$ and all $\cal L$-nodes have multiplicity one.
		\item \label{thm:main_L-vector} 
		The maximal ideal divisor $Z_{\max}(X,0)$ of $(X,0)$ coincides with the fundamental cycle $Z_{\min}$ of $\Gamma_\pi$. In particular, $\Gamma_\pi$ determines the multiplicity $m_v$ associated with every vertex $v$ of $\Gamma_\pi$, and therefore also the set $V_{\cal L}$ of $\cal L$-nodes of $\Gamma_\pi$, the $\cal L$-vector $L_\pi$ of $(X,0)$, and the distance $d$ on $\Gamma_\pi$.
		
		\item \label{thm:main_inner_rates} 
		The inner rate $q_v$ of each vertex (or, more generally, of each {divisorial point}) of $\Gamma_\pi$ is given by 
		\[
		q_v=d(v, V_{\cal L})+1.
		\]
		
		\item \label{thm:main_P-vector} 
		The $\cal P$-vector $P_\pi$ of $(X,0)$ is determined by
		$$p_v = - E_v \cdot \Big(\sum\nolimits_{v'} (m_{v'} q_{v'}-1) E_{v'} - (Z_{\Gamma_{\pi}} - Z_{\min})\Big)$$
		for every vertex $v$ of $\Gamma_\pi$.
		Moreover, if $v$ is an $\cal L$-node, this formula simplifies to $p_v=2\big(g(E_v)+l_v-1\big)$.
				
		\item \label{thm:main_P-nodes} 
		Let $\pi'$ be the minimal good resolution of $(X,0)$ that factors through its Nash transform.
		A vertex $v$ of $\Gamma_{\pi'}$ is a $\cal P$-node of $(X,0)$ if and only if either $l_v>1$ or there exist two distinct vertices $v'$ and $v''$ of $\Gamma_{\pi'}$ adjacent to $v$ and such that $q_{v'},q_{v''}<q_v$. 
		
		\item \label{thm:main_Nash_transform}
		The resolution $\pi'$ is obtained by composing $\pi$ with a finite sequence of blowups of double points of the respective exceptional divisor: at each step, a double point in $E_v\cap {E_{v'}}$ has to be blown up if and only if $|q_v-{q_{v'}}|<d(v,v')$. 
		In particular, an edge $e=[v,v']$ of $\Gamma_\pi$ contains a $\cal P$-node of $(X,0)$ in its interior (that is, such a $\cal P$-node appears as a vertex after blowing up finitely many double points of the exceptional divisor, starting with the blowup of the double point associated with $e$) if and only if $|q_v-{q_{v'}}|<d(v,v')$; when this is the case, $e$ contains exactly one $\cal P$-node $w$, and the inner rate of $w$ is $q_w=\big(d(v,v')+q_v+q_{v'}\big)/2$. 
			\end{enumerate}
\end{theorem}

In particular, we can build from $\Gamma_{\pi}$ the resolution graph $\Gamma_{\pi'}$ of $\pi'$, decorated by arrows corresponding to the components of the polar curve of a generic plane projection $\ell \colon (X,0) \to (\C^2,0)$ and by the inner rate of each vertex.  
Observe also that \ref{thm:main_factorization_blowup} and \ref{thm:main_L-vector} imply that the multiplicity of a LNE normal germ is determined by its topological type.

While the first two parts of the Theorem are quite elementary, the remaining parts rely heavily on a careful study of generic projections of LNE surfaces (see Lemma~\ref{lem:projections_LNE}), building on results from \cite{NeumannPedersenPichon2020a}.
Parts \ref{thm:main_inner_rates} and \ref{thm:main_P-vector} also depend on the study of inner rates of \cite{BelottodaSilvaFantiniPichon2019}, and in particular on the so-called Laplacian Formula of \emph{loc.\ cit.}

\medskip
We then move our attention to the study of the discriminant curve $\Delta$ of a generic plane projection $\ell\colon(X,0)\to(\C^2,0)$ of $(X,0)$.
Our second main result, which generalizes Bondil's results \cite[Theorem 4.1]{Bondil2003}, \cite[Proposition 5.4]{Bondil2016}, can be stated as follows.

\begin{theorem} \label{thm:main2} 
	Let $(X,0)$ be a LNE normal surface germ and let $\pi\colon X_\pi \to X$ be the minimal good resolution of $(X,0)$. 
	Then the dual graph $\Gamma_\pi$ of $\pi$ determines the embedded topological type of the discriminant curve of a generic plane projection $\ell\colon(X,0)\to(\C^2,0)$ of $(X,0)$.
\end{theorem}

To be more precise, the embedded topological type of a plane curve can be conveniently encoded in a combinatorial object, its \emph{Eggers--Wall tree}, whose construction will be recalled in Section~\ref{sec:discriminant} (see also \cite[Definition~3.9]{BarrosoPerezPopescu2019}).
We will give a more precise statement of Theorem~\ref{thm:main2} in Theorem~\ref{thm:discriminant_precise}, showing explicitly how to obtain the Eggers--Wall tree of the discriminant curve $\Delta$ of a generic plane projection $\ell\colon(X,0)\to(\C^2,0)$ of $(X,0)$ as the quotient of the graph $\Gamma_{\pi'}$ by a suitable equivalence relation.

\medskip

Part~\ref{thm:main_P-vector} of Theorem~\ref{thm:main} can be thought of as the uniqueness of a solution, within the class of LNE surface singularities, to what we refer to as the problem of \emph{polar exploration} of surface singularities, which asks to determine the possible configurations of arrows of a finite graph that can be realized as polar curves of a complex surface germ $(X,0)$. 
Recall that surface singularities can be resolved either by a sequence of normalized point blowups, following seminal work of Zariski \cite{Zariski1939} from the late nineteen thirties, or by a sequence of normalized Nash transforms, as was done half a century later by Spivakovsky \cite{Spivakovsky1990}. 
The relationship between these two resolution algorithms, and therefore between hyperplane sections and polar curves of a surface singularity, is still quite mysterious, and they seem to be in some sense dual, as was observed by L\^e \cite[\S4.3]{Le2000}.

More precisely, recall that the incidence matrix of the dual graph $\Gamma_\pi$ associated with a good resolution $\pi\colon X_\pi\to X$ of $(X,0)$ is negative definite by a classical result of Mumford \cite[\S1]{Mumford1961}.
Moreover, Grauert \cite{Grauert1962} proved that every weighted graph $\Gamma$ without loops and with negative definite incidence matrix can be realized as dual graph $\Gamma_\pi$ associated with a good resolution of some normal complex surface germ $(X,0)$.
It is well known that the weighted graph $\Gamma_\pi$ determines the topology of $(X,0)$, since $\Gamma_\pi$ is a plumbing graph of the link of $(X,0)$, and conversely, as we have already mentioned, Neumann \cite{Neumann1981} proved that the plumbing graph $\Gamma_\pi$ is determined up to a natural equivalence relation by the topology of the surface germ.
It is thus natural to consider the plumbing graph $\Gamma_\pi$ endowed with an $\cal L$- and a $\cal P$-vector.
From this point of view, our result implies the following statement.

\begin{corollary} \label{cor:exploration}  
	Let $\Gamma$ be a finite connected graph without loops weighted by attaching to each vertex $v$ a genus $g(v)\geq 0$ and a self-intersection $e(v)<0$.
	Then there exists at most one pair $(L,P)$ of vectors $L=(l_v)$ and $P=(p_v)\in (Z_{\geq0})^{V(\Gamma)}$ such that there exist a LNE normal surface singularity $(X,0)$ and a good resolution $\pi$ of $(X,0)$ satisfying
	\[
	(\Gamma,L,P)=(\Gamma_\pi, L_\pi, P_\pi).
	\]
\end{corollary}

Observe that not all weighted graphs can be realized as resolution graph of a LNE surface germ.
For instance, one topological restriction comes from parts~\ref{thm:main_factorization_blowup} and \ref{thm:main_L-vector} of Theorem~\ref{thm:main}: if $\Gamma$ is the resolution graph of a LNE normal surface germ, then for every vertex $v$ of $\Gamma$ such that $E_v\cdot Z_{\min}<0$ the component $E_v$ has to have multiplicity one in $Z_{\min}$.

\subsection*{Acknowledgments}
We thank Patrick Popescu-Pampu and Bernard Teissier for interesting discussions about quasi-ordinary singularities and generic plane projections, and Javier Fern\'andez de Bobadilla for giving us an argument that simplified the proof of Proposition~\ref{proposition:appendix_smooth_transverse}.
We would also like to thank Camille Le Van and Delphine Menard for fruitful conversations.
This work has been partially supported by the project \emph{Lipschitz geometry of singularities (LISA)} of the \emph{Agence Nationale de la Recherche} (project ANR-17-CE40-0023) and by the 
\emph{PEPS--JCJC M\'etriques singuli\`eres, valuations et g\'eom\'etrie Lipschitz des vari\'et\'es} of the \emph{Institut National des Sciences Math\'ematiques et de leurs Interactions} of the \emph{Centre National de la Recherche Scientifique}. 
The second author has also been supported by a \emph{Research Fellowship} of the \emph{Alexander von Humboldt Foundation}.


\section{Surface germs with unique $\mathcal L$-vector}
\label{sec:Zmin=Zmax}

In this section we prove parts \ref{thm:main_factorization_blowup} and \ref{thm:main_L-vector} of Theorem~\ref{thm:main}.
More generally, we are interested in finding a suitable geometric condition yielding a class of complex surfaces $(X,0)$ whose $\cal L$-vector is completely determined by the topology of a resolution.
In order to achieve this, we recall the precise definitions of the divisors $Z_{\max}(X,0)$ and $Z_{\min}$ that have been mentioned in the introduction, and determine a condition that guarantees their equality.

\medskip

We begin by recalling the notion of Lipman cone.
A more thorough discussion of the objects described in this section can be found in \cite{Nemethi1999}.
Let $\Gamma$ be a finite connected graph without loops and such that each vertex $v\in V(\Gamma)$ is weighted by two integers $g(v)\geq0$, called genus, and $e(v)$, called self-intersection.
We assume that the incidence matrix induced by the self-intersections of the vertices of $\Gamma$, that is the matrix $I_\Gamma\in \Z^{V(\Gamma)}$ whose $(v,v')$-th entry is $e(v)$ if $v=v'$, and the number of edges of $\Gamma$ connecting $v$ to $v'$ otherwise, is negative definite.	
Let $E=\bigcup_{v\in V(\Gamma)}E_v$ be a configuration of curves whose dual graph is $\Gamma$, so that $I_{\Gamma} = (E_v \cdot E_{v'})$, and consider the free additive group $\cal G$ generated by the irreducible components of $E$, that is 
\[
\cal G = \bigg\{D =\sum_{v \in V(\Gamma)} m_v E_v \,\bigg|\, m_v \in \Z\bigg\}. 
\]
By a slight abuse of notation, we refer to the elements of $\cal G$ as \emph{divisors on $\Gamma$}.
On $\cal G$ there is a natural intersection pairing $D\cdot D'$, described by the incidence matrix $I_{\Gamma}$, and a natural partial ordering given by setting $\sum m_v E_v \leq \sum m'_v E_v$ if an only if $ m_v \leq m'_v$ for every $v \in V(\Gamma)$. 

The {\it Lipman cone} of $\Gamma$ is the semi-group $\cal E^+$ of $\cal G$  defined as 
\[
\cal E^+ = \big\{D \in \cal G \,\big|\, D \cdot E_v \leq 0  \text{ for all } v \in V(\Gamma) \big\}.
\]

\begin{remark}
	By looking at the coefficients of a divisor we can identify $\cal G$ with the additive group $\Z^{V(\Gamma)}$.
	Then the Lipman cone $\cal E^+$ of $\Gamma$ is naturally identified with the cone $\Z_{\geq0}^{V(\Gamma)}\cap-I_\Gamma^{-1}\big(\Q_{\geq0}^{V(\Gamma)}\big)$, since by definition a divisor $\sum m_vE_v$ belongs to $\cal E^+$ if and only if the vector $I_\Gamma\cdot(m_v)_{v\in V(\Gamma)}$ belongs to $\Z_{\leq0}^{V(\Gamma)}$.
\end{remark}

A cardinal property of the Lipman cone $\cal E^+$, proven in \cite[Proposition 2]{Artin1966}, is that it has a unique nonzero minimal element $Z_{\min}$, called the {\it fundamental cycle} of $\Gamma$, and that moreover $Z_{\min}>0$, that is the coefficients of $Z_{\min}$ are all positive.
Observe that the existence of the fundamental cycle and the fact that $Z_{\min}>0$ are equivalent to the fact that $D>0$ for every nonzero divisor $D$ in $\cal E^+$.

Assume from now on that $\Gamma$ is the dual graph of a good resolution of a normal surface singularity $(X,0)$.
Notice that the Lipman cone, and therefore its fundamental cycle, only depend on the graph $\Gamma$, that is on the topology of $(X,0)$, and not on the complex geometry of $(X,0)$; the fundamental cycle $Z_{\min}$ can be explicitly computed from $\Gamma$ by using Laufer's algorithm from \cite[Proposition 4.1]{Laufer1972}.

Consider now a germ of analytic function $f \colon (X, 0) \to (\C, 0)$. 
The \emph{total transform} of $f$ by $\pi$  is the divisor $(f) = (f)_{\Gamma} + f^*$ on $X_\pi$, where $f^*$ is the strict transform of $f$ and $(f)_{\Gamma} = \sum_{v \in V(\Gamma)} m_v(f) E_v$ is the divisor supported on $E$ such that $m_v(f)$ is the multiplicity of $f \circ \pi$ along $E_v$. 
By \cite[Theorem 2.6]{Laufer1971}, we have 
\begin{equation}\label{eq:IdentityTotalTransform}
(f) \cdot E_v=0 \quad \text{for all }v \in V(\Gamma).
\end{equation}
In particular, $(f)_\Gamma$ belongs to the Lipman cone $\cal E^+$ of $\Gamma$, and therefore the semi-group $\cal A_X^+ = \{ (f)_{\Gamma} \;|\; f \in \cal \cal O_{(X,0)}\}$\label{def:A+} of $\cal G$ is contained in $\cal E^+$; it has a unique nonzero minimal element $Z_{\max}(X,0)$, which is called the {\it maximal ideal divisor} of $(X,0)$. 
%
%
Observe that the divisor $Z_{\max}(X,0)$ coincides with the cycle $(h)_{\Gamma}$ of a generic linear form $h \colon (X,0) \to (\C,0)$, and that by the definition of the fundamental cycle we have $Z_{\min} \leq  Z_{\max}(X,0)$.

The following proposition is the main result of this section.

\begin{proposition} \label{prop:reduced tangent cone}
	Let $(X,0)$ be a normal surface singularity and let $\pi \colon (X_{\pi},E) \to (X,0)$ be the minimal good resolution of $(X,0)$.
	If a generic hyperplane section of $(X,0)$ is a union of smooth curves, then:
	\begin{enumerate}
		\item $\pi$ factors through the blowup of the maximal ideal of $(X,0)$ and all $\cal L$-nodes have multiplicity one;
		\item the maximal ideal divisor $Z_{\max}(X,0)$ of $(X,0)$ coincides with the fundamental cycle $Z_{\min}$ of $\Gamma_\pi$.
	\end{enumerate}
\end{proposition}

\begin{proof}
	Let $\pi' \colon X_{\pi'} \to X$ be the minimal good resolution of $(X,0)$ which factors through the blowup of its maximal ideal and let $E_v$ be a component of $(\pi')^{-1}(0)$. 
	Let $\gamma^*$ be a  curvette of  $E_v$, that is a smooth complex curve germ intersecting transversely $E_v$ at a smooth point of $(\pi')^{-1}(0)$, and let $h' \colon (X,0) \to (\C,0)$ be a generic linear form of $(X,0)$ such that the strict transform of $(h')^{-1}(0)$ via $\pi'$ does not pass through the point $p=\gamma^* \cap E_v$. 
	Then the multiplicity  $\mult(\gamma,0)$ of $\gamma=\pi'(\gamma^*)$ at $0$ can be computed as the intersection multiplicity of $\gamma$  with a Milnor fiber $\{h=t\}$ of $h$ in a small neighborhood of $0$.  
	Let us choose local coordinates $(u,v)$ centered at $p$ such that $u=0$ is a local equation for $E_v$ and $v=0$ a local equation for $\gamma^*$.
	Then by definition of $m_v$ we have $ (h' \circ \pi) (u,v) = u^{m_v} \alpha(u,v)$ where $\alpha(u,v)$ is a unity in $\C\{u,v\}$, and therefore $m_v  =  \mult(\gamma,0)$.
	
	If $v$ is an $\cal L$-node of $(X,0)$ and $h \colon (X,0) \to (\C,0)$ is a generic linear form of $(X,0)$, so that $h^{-1}(0)$ is a generic hyperplane section of $(X,0)$, then there exists an irreducible component $\gamma$ of $h^{-1}(0)$ whose strict transform $\gamma^*$ by $\pi'$ intersects $E_v$.  
	By hypothesis the curve $\gamma$ is smooth, therefore it has multiplicity $1$ and $\gamma^*$ is a curvette of $E_v$.
	This proves that $m_v=1$.
	
	Assume now  that  $\pi$ does not factor through the blowup of the maximal ideal, so that $\pi'= \pi \circ \alpha$, where $\alpha$ is a finite composition of point blowups.
	By minimality of $\pi'$ there exists an $\cal L$-node $v_0$ of $(X,0)$ which is associated with the exceptional component of one of the point blowups in $\alpha$.
	Let $\alpha_1$ be the first blowup in the sequence $\alpha$, that is, $\alpha_1$ is the blowup of $X_\pi$ at a point $p$ of $E_v \cap h^*$, where $E_v$ is a component of $\pi^{-1}(0)$, and let $E_{w}$ be the exceptional curve of $\alpha_1$.
	Since $h^*$ passes through $p$, we have $m_w=m_{w}(h) > m_v(h)\geq1$. 
	As this argument can be repeated for every blowup forming $\alpha$, we deduce that $m_{v_0}(h)>1$ as well, contradicting the first part of the proof. 
	This implies that $\alpha$ must be an isomorphism, proving $(i)$.

	To prove $(ii)$, write $Z_{\max}=Z_{\max}(X,0)=\sum_{v\in V(\Gamma)}m_vE_v$ and $Z_{\min}=\sum_{v\in V(\Gamma)}\widetilde m_vE_v$, and for every $v$ in $V(\Gamma)$ consider the non-negative integers
	\[
	l_v=-Z_{\max}\cdot E_v
	\,\,\,\,\,\,\text{ and }\,\,\,\,\,\,
	{\widetilde{l}_v}=-Z_{\min}\cdot E_v\,.
	\]
	Since $Z_{\min}\leq Z_{\max}$ by definition of $Z_{\min}$, it is enough to prove that $Z_{\min}\geq Z_{\max}$. 
	Since $I_\Gamma$ is negative definite, it is therefore sufficient to show that the integer $(Z_{\min}-Z_{\max})\cdot E_v =  l_v-{\widetilde{l}_v} $ is  at most zero for every vertex $v$ of $\Gamma$. 
	Whenever $l_v = 0$, this follows immediately from the definition, so let us fix a vertex $v$ such that $l_v >0$. 
	From part $(i)$, we know that $m_v=1$. It follows from the inequality $0<\widetilde{m}_v \leq m_v =1$ that $\widetilde{m}_v =1$ as well. 
	We therefore get:
	\[
	l_v - {\widetilde{l}}_v = (Z_{\min}-Z_{\max})\cdot E_v = \sum_{w\in V(\Gamma)}(\widetilde m_{w}-m_w)E_w\cdot E_v = \sum_{w\neq v}(\widetilde m_{w}-m_w)E_w\cdot E_v  \leq 0,
	\]
	since $E_w\cdot E_v \geq 0$ whenever $w\neq v$ and $\widetilde m_{w} \leq m_w$ at all vertices.
\end{proof}

The hypothesis of Proposition~\ref{prop:reduced tangent cone} is quite weak, as it is satisfied by every normal surface germ with \emph{reduced tangent cone} (in which case the components of a generic hyperplane section are not only smooth but also transverse, see for example \cite[\S1]{GonzalezLejeune1997}), for example by every minimal surface singularity.
More generally, the hypothesis holds for all LNE surface germs, as was proven in \cite[Theorem~3.10]{FernandesSampaio2019}.
In particular, the proposition implies parts \ref{thm:main_factorization_blowup} and \ref{thm:main_L-vector} of Theorem~\ref{thm:main}.

	\medskip

	Observe that it follows by equation \eqref{eq:IdentityTotalTransform} that the vector $-I_{\Gamma_\pi}\cdot Z_{\max}(X,0)$ of $\Z^{V(\Gamma_\pi)}_{\geq0}$ coincides with the $\cal L$-vector $L_\pi$ of $(X,0)$ considered in the introduction.
	Therefore, whenever $Z_{\max}(X,0)$ is determined by the topological type of $(X,0)$, the same holds true for $L_\pi$.
	We collect this result, which is the first step towards the proof of Corollary~\ref{cor:exploration}, in the following corollary.

\begin{corollary}
	Let $\Gamma$ be a weighted graph.
	Then there exists at most one vector $L\in \Z^{V(\Gamma)}$ such that there exist a normal surface germ $(X,0)$ whose generic hyperplane section is a union of smooth curves and a good resolution $\pi\colon X_\pi \to (X,0)$ of $(X,0)$ satisfying $(\Gamma,L)=(\Gamma_\pi,L_\pi)$.
\end{corollary}


\section{A lemma on generic projections} \label{sec:lemma generic}

In this section we will introduce three notions that will prove fundamental in the remaining part of the paper, namely generic projections, non-archimedean links, and local degrees.
We will also prove an important result, Lemma~\ref{lem:generic projection}, that shows the compatibility of generic projections with minimal resolutions.

\medskip

We begin by discussing the notion of generic projection, which is based on seminal work of Teissier.
Fix an embedding of $(X,0)$ in a smooth germ $(\C^n,0)$, and consider the morphism $\ell_{\cal D}\colon(X,0)\to(\C^2,0)$ obtained as the restriction to $X$ of the projection along an $(n-2)$-dimensional linear subspace $\cal D$ of $\C^n$. 
Recall that whenever $\ell_{\cal D}$ is finite, the associated \emph{polar curve} $\Pi_{\cal D}$ is the closure in $(X,0)$ of the ramification locus of the restriction of $\ell_{\cal D}$ to $X\setminus\{0\}$, and the associated \emph{discriminant curve} is the plane curve $\Delta_{\cal C}=\ell_{\cal D}(\Pi_{\cal D})$. 
The Grassmannian variety $\Gr(n-2,\C^n)$ of $(n-2)$-planes in $\C^n$ contains an analytic dense open subset $\Omega$ such that for every $\cal D$ in $\Omega$ the projection $\ell_{\cal D}$ is finite and the families $\{\Pi_{\cal D}\}_{\cal D\in\Omega}$ and $\{\Delta_{\cal D}\}_{\cal D\in\Omega}$ are both well behaved (for example, they are equisingular in a strong sense).
We say that a morphism $\ell\colon (X,0)\to(\C^2,0)$ is a \emph{generic projection} of $(X,0)$ if $\ell=\ell_{\cal D}$ for some $\cal D$ in $\Omega$.
A discussion of the properties satisfied by a generic projection, leading to a precise definition of $\Omega$, can be found in \cite[\S2]{NeumannPedersenPichon2020a}, building on work of Teissier (see in particular \cite[Lemme-cl\'e V 1.2.2]{Teissier1982}); we will come back to this matter later in this section.

\medskip

We now recall the definition of the non-archimedean link $\NL(X,0)$ of the germ $(X,0)$.
Indeed, our goal for this section is to study the map induced by a generic projection $\ell\colon (X,0)\to(\C^2,0)$ on the dual graph of a good resolution of $(X,0)$.
In principle, for this to make sense it is necessary to chose a suitable good resolution $\pi\colon X_\pi\to X$ of $(X,0)$ and a compatible sequence of point blowups $\sigma\colon Y_\sigma\to\C^2$ of $(\C^2,0)$ in order for $\ell$ to induce a map $|\Gamma_\pi|\to|\Gamma_\sigma|$ between the topological spaces underlying $\Gamma_\pi$ and $\Gamma_\sigma$. 
In this paper we will use $\NL(X,0)$ as a convenient way of encoding intrinsically all the dual graphs of good resolutions of $(X,0)$; for this purpose, we can adopt the following ad hoc definition.
Recall that, if $\pi \colon X_\pi \to X$ and $\pi' \colon X_{\pi'} \to X$ are two good resolutions of $(X,0)$ such that $\pi'$ dominates $\pi$ (that is, $\pi'$ factors through $\pi$), then we have a natural inclusion $|\Gamma_\pi|\hookrightarrow |\Gamma_{\pi'}|$ between the topological spaces underlying the dual graphs $\Gamma_\pi$ and $\Gamma_{\pi'}$, and a retraction $|\Gamma_{\pi'}|\to |\Gamma_\pi|$ obtained by contracting the trees in $|\Gamma_{\pi'}|\setminus |\Gamma_{\pi}|$.
The non-archimedean link can then be seen as the inverse limit $\NL(X,0)=\varprojlim_{\pi}|\Gamma_\pi|$ in the category of topological spaces and with respect to the various retraction morphisms, where the limit runs over the poset of good resolutions of $(X,0)$, ordered by domination.
In particular, $\NL(X,0)$ contains a copy the dual graph of each good resolution of $(X,0)$, and it can be seen as a compactification of the infinite union $\bigcup_\pi|\Gamma_\pi|$ of all the dual graphs of the good resolutions of $(X,0)$.
As such, it can be thought of as a universal dual graph of the singularity $(X,0)$.
To unburden the notation, in the remaining part of the paper we will usually identify a dual graph $\Gamma_\pi$ with its image $|\Gamma_\pi|$ in $\NL(X,0)$.
This point of view makes it convenient to think of $\cal L$- and $\cal P$-nodes abstractly as points of $\NL(X,0)$.

Traditionally, the non-archimedean link $\NL(X,0)$ is built as a space of normalized semivaluations on the complete local ring $\widehat{\cal O_{X,0}}$ of $X$ at $0$.
In particular, if $\pi\colon X_\pi\to X$ is a good resolution of $(X,0)$ and $E_v$ is a component of its exceptional divisor $\pi^{-1}(0)$, the corresponding vertex of $\Gamma_\pi$ is identified with the corresponding \emph{divisorial valuation} $v\colon \widehat{\cal O_{X,0}}\to\R_+\cup\{+\infty\}$ defined by $v(f)=\ord_{E_v}(\pi^*f)/m_v$, where $\ord_{E_v}(\pi^*f)$ denotes the order of vanishing along $E_v$ of the pullback of $f$ via $\pi$.
Throughout the paper, we will freely make use of this terminology, calling \emph{divisorial point} of $\NL(X,0)$ (or of a given dual graph $\Gamma_\pi$) any point that can arise in this way, and denoting by $E_v$ any exceptional curve corresponding to a divisorial point $v$.
Observe that the subset of $\NL(X,0)$ consisting of its divisorial points is dense in the non-archimedean link; this corresponds to the fact that any given dual graph $\Gamma_\pi$ can be refined \emph{ad infinitum} by passing to resolutions dominating $\pi$, subdividing each edge $e=[v,v']$ into smaller edges by successively blowing up double points starting with the blowup of $X_\pi$ at the point of $E_v\cap {E_{v'}} $ that corresponds to $e$.
In particular, a divisorial point of $\NL(X,0)$ is contained in the interior of $e$ if and only if it is associated with an exceptional component that appears after blowing up only double points as above.
We refer the reader to \cite[\S2.1]{BelottodaSilvaFantiniPichon2019} and \cite{Fantini2018} for further details on this point of view.

\medskip

The morphism $\ell\colon (X,0)\to (\C^2,0)$ induces a natural map $\widetilde\ell\colon\NL(X,0)\to\NL(\C^2,0)$.
From the point of view of semivaluations, this is simply defined functorially by pre-composing a semivaluation on $\widehat{\cal O_{X,0}}$ with the morphism of complete local rings $\widehat{\cal O_{\C^2,0}} \to \widehat{\cal O_{X,0}}$ induced by $\ell$.

Concretely, $\widetilde\ell(v)$ can also be computed explicitly on a divisorial point $v$ of $\NL(X,0)$ as follows: we can find a sequence of point blowups $\sigma_{\ell,v}\colon Y_{\sigma_{\ell,v}} \to \C^2$ of $(\C^2,0)$ and a good resolution $\pi_{\ell,v}\colon X_{\pi_{\ell,v}}\to X$ of $(X,0)$ such that $v$ corresponds to a component $E_v$ of the exceptional divisor of $\pi_{\ell,v}$, the composition $\ell \circ \pi_{\ell,v} \colon X_{\pi_{\ell,v}} \to \C^2$ factors through a map $\widehat\ell \colon X_{\pi_{\ell,v}} \to Y_{\sigma_{\ell,v}}$ making the following diagram commute
\begin{equation}\label{eqn:diagram_defining_local_degree}
\xymatrix@C=5em{    
	X_{\pi_{\ell,v}}         \ar[d]_{\widehat\ell}    \ar[r]^{\pi_{\ell,v}} & X    \ar[d]^{\ell}           \\
	Y_{\sigma_{\ell,v}}      \ar[r]_{\sigma_{\ell,v}}  &  \C^2
}
\end{equation}
and such that $E_v$ is mapped by $\widehat\ell$ surjectively onto a component $E_w$ of the exceptional divisor of $\sigma_{\ell,v}$; we then have $\widetilde\ell(v)=w$.
	
Let $\pi\colon X_\pi\to X$ be a good resolution of $(X,0)$.
Since $\ell$ is a finite map ramified precisely over the associated polar curve, the induced map $\widetilde\ell|_{\Gamma_\pi}\colon\Gamma_\pi\to\widetilde\ell({\Gamma_\pi})$ is itself a finite cover, which on a set contained in the the set of $\cal P$-nodes of $(X,0)$ that are contained in $\Gamma_\pi$ (not necessarily as vertices but possibly in the interior of some edges).
In particular, $\widetilde\ell$ cannot contract an edge of $\Gamma_\pi$, but it may fold one if it contains a $\cal P$-node in its interior.

\medskip

Observe that the map $\widetilde\ell$ clearly depends on the choice of $\ell$.
Indeed, if $\ell'\colon (X,0)\to(\C^2,0)$ is another generic projection obtained by composing $\ell$ with an automorphism $\phi$ of $(\C^2,0)$, then $\phi$ induces a nontrivial automorphism $\widetilde\varphi$ of $\NL(\C^2,0)$, and we have $\widetilde\ell = \widetilde\varphi \circ \widetilde\ell'$.
While in general two generic projections of $(X,0)$ do not differ by an automorphism of $(\C^2,0)$, it it possible to control this phenomenon if we restrict $\widetilde\ell$ to the dual graph $\Gamma_\pi$ of the minimal good resolution $\pi \colon X_\pi \to X$ of $(X,0)$ that factors through its Nash transform, as we explain in Lemma~\ref{lem:generic projection} below.

In order to do this, we need to dive deeper into the definition of generic projections, to be able to study the polar curves and the discriminant curves of $(X,0)$ in families.
Let us begin by recalling the precise notion of strong equiresolution of singularities given in \cite[3.1.1 and 3.1.5]{Teissier1976}.
Given a morphism $\beta \colon M \to \Lambda$ with reduced fibers between smooth connected complex manifolds and a simple normal crossing divisor $E$ of $M$, we say that $\beta$ is \emph{simple} (with respect to $E$) if $\beta$ is smooth and its restriction $\beta|_E \colon E \to \Lambda$ to $E$ is proper and locally a trivial deformation along its fibers.
If we have another morphism $\sigma \colon M' \to M$, we say that $\sigma$ is \emph{$\beta$-compatible} if the composition $\beta'=\beta \circ \sigma$ is simple (with respect to $E'=\sigma^{-1}(E)$).
Finally, given a (singular) subvariety $X$ of $M$, we say that an embedded resolution of singularities $\pi \colon \widetilde{M} \to M$ of $X$ is a \emph{strong equiresolution} (along $\Lambda$) of $X$ if $\pi$ is $\beta$-compatible and all of its restrictions $\pi_\lambda$ over $\lambda\in\Lambda$ are good embedded resolutions of $X_\lambda$.

According to \cite[Lemme-cl\'e V 1.2.2]{Teissier1982} (see \cite[Proposition~2.3]{NeumannPedersenPichon2020a} for an English presentation), there exists an analytic open dense subset $\Omega$ of the Grassmannian $\Gr(n-2,\C^n)$ where the family $\{(\Delta_{\cal D},\cal D)\}_{\cal D\in\Omega}$ of discriminant curves, which can be seen as a surface in $(\mathbb{C}^2,0)\times \Omega$ fibered over $\Omega$ via the projection $\beta\colon(\mathbb{C}^2,0)\to\Omega$ on the second factor, admits a strong equiresolution
\[
\xymatrix@C=5em{    
	(\mathcal{Y},\cal F)   \ar[r]^{\sigma} \ar[rd]_{ \beta_{\mathcal{Y}} } & \ar[d]^{\beta}   (\C^2,0) \times \Omega  \\
	&\Omega         &  \\
}
\]
with $F=\sigma^{-1}\big(\{0\}\times\Omega\}\big)$ a simple normal crossing divisor of $\mathcal{Y}$. 

For each $\cal D$ in $\Omega$, denote by $\sigma_{\cal D}\colon (\mathcal{Y}_{\cal{D}}, F_{\cal{D}})\to(\C^2,0)$ the restriction of $\sigma$ to the fiber $\beta_{\cal Y}^{-1}(\cal D)$, which is a sequence of point blowups of $(\C^2,0)$. 
Given two elements $\cal D$ and $\cal D'$ of $\Omega$, this allows us to define an isomorphism of graphs $\eta_{\cal D,\cal D'}\colon \Gamma_{\sigma_{\cal D}} \stackrel{\sim}{\longrightarrow} \Gamma_{\sigma_{\cal D'}}$ as follows.
For each $v \in V(\Gamma_{\sigma_{\cal D}})$, if we denote by $F^{\cal D}_{v}$ the corresponding irreducible component of $\cal F=\sigma_{\cal D}^{-1}(0)$, there is a unique irreducible component $\cal F_v^{\cal D}$ of $\sigma^{-1}(\{0\} \times \Omega)$ such that $F^{\cal D}_{v} =  \cal F_v^{\cal D} \cap \sigma_{\cal D}^{-1}(0)$.
We then set $\eta_{\cal D,\cal D'}(v)=v'$, where $v'$ is the vertex of $\Gamma_{\cal D'}$ such that $\cal F_{v'}^{\cal D'}=\cal F_v^{\cal D}$ (that is, equivalently, such that $\cal F_v^{\cal D} \cap \sigma_{\cal D'}^{-1}(0)=F^{\cal D'}_{v'}$).
This yields a bijection $V(\Gamma_{\sigma_{\cal D}}) \to V(\Gamma_{\sigma_{\cal D'}})$ which extends to a natural homeomorphism
\[
{\eta}_{\cal D, \cal D'} \colon \Gamma_{\sigma_{\cal D}} \to  \Gamma_{\sigma_{\cal D'}}
\]
defined on the divisorial points of $\Gamma_{\sigma_{\cal D}}$ as follows. Fix $\cal D \in \Omega$ and consider a divisorial point $v$ on an edge $[v_1,v_2]$ of $\Gamma_{\sigma_{\cal D}}$. 
Then $E^{\cal D}_v$ is created by a finite sequence of blowups of double points of the previous exceptional divisor, starting with the blowup of the point $F_{v_1}^{\cal D} \cap F_{v_2}^{\cal D}$. 
We can perform this blowups in family by blowing up along successive intersections of the form $\cal F_{w_1}^{\cal D} \cap \cal F_{w_2}^{\cal D}$, starting with the blowup along $\cal F_{v_1}^{\cal D} \cap \cal F_{v_2}^{\cal D}$. 
By composing this sequence of blowups with $\sigma$, we obtain a ($\beta$-compatible) morphism $\sigma_v \colon {\cal Y}_{\sigma_v} \to (\C^2,0) \times \Omega$. 
The last blowup creates an irreducible new component $\cal F_v^{\cal D}$ in the exceptional divisor, and as before we define $v' = \eta_{\cal D, \cal D'}(v)$ by declaring that the corresponding irreducible component $F_{v'}^{\cal D'}$ is the intersection $\cal F_v^{\cal D} \cap ({\sigma'})^{-1}(0, \cal D')$.
Observe that, since multiplicities are constant along a smooth family,  we have $m_v=m_{\eta_{\cal D,\cal D'}(v)}$ for every divisorial point $v$ of $\Gamma_{\sigma_{\cal D}}$.

The following lemma relating the graph $\Gamma_{\sigma_{\cal D}}$ to $\Gamma_{\pi}$ plays a crucial role in several arguments in the rest of the paper.

\begin{lemma}  \label{lem:generic projection}
	Let $(X,0)$ be a normal surface singularity, let $\pi\colon X_\pi\to X$ be the minimal good resolution of $(X,0)$ that factors through the blowup of  its maximal ideal and its Nash transform, and let $\Gamma_\pi\subset \NL(X,0)$ be the dual graph of $\pi$. 
	Then for all $\cal D$ and $\cal D'$ in $\Omega$ the diagram
	\[
	\xymatrix{    
		&\Gamma_{\pi}    \ar[ld]_{\widetilde\ell_{\cal D}|_{\Gamma_\pi} }  \ar[rd]^{\widetilde\ell_{\cal D'}|_{\Gamma_\pi}}         &  \\
		\Gamma_{\sigma_{\cal D}}   \ar[rr]_{{\eta}_{\cal D, \cal D'}}  & &   \Gamma_{\sigma_{\cal D'}}
	}
	\]
	obtained by restricting to the graph $\Gamma_{\pi}$ the two induced morphisms of non-archimedean links $\widetilde\ell_{\cal D},\widetilde\ell_{\cal D'} \colon \NL(X,0) \to \NL(\C^2,0)$, is commutative.
\end{lemma}

Before moving to the proof of the lemma, which is rather technical, we observe that the homeomorphism ${\eta}_{\cal D, \cal D'} \colon \Gamma_{\sigma_{\cal D}} \to  \Gamma_{\sigma_{\cal D'}}$ lifts naturally to an automorphism ${\eta}_{\cal D, \cal D'}$ of the dual graph $\Gamma_{\pi'}$ of any good resolution $\pi'\colon X_{\pi'}\to X$ of $(X,0)$.
%
%
However, the commutativity ${\eta}_{\cal D, \cal D'}  \circ \widetilde\ell_{\cal D}$ does not necessarily hold on the whole of $\Gamma_{\pi'}$.
We defer an illustration of this phenomenon to Example~\ref{ex:pi_not_minimal_multiplicities}, since showing this now would require a lengthy local computation, while after proving Lemma~\ref{lem:projections_LNE} we can give a more conceptual explanation.

As our needs go slightly beyond what was done by Teissier, let us explain how to adapt his constructions accordingly.
We start by proving a technical lemma about resolution in families of surfaces, much in the spirit of \cite[4.1 and 4.2]{Teissier1976}:

\begin{lemma}\label{cl:TechBlowingup}
	Let $M$ and $\Omega$ be connected complex manifold such that $\mbox{dim}(M) =\mbox{dim}(\Omega)+2$, let $E$ be a simple normal crossing divisor of $M$, and let $\beta: M \to \Omega $ be a simple morphism (with respect to $E)$. 
	Consider a finite sequence of (adapted) smooth blowups $\sigma \colon (M',E') \to (M,E)$ whose centers have codimension at least $2$. 
	Then, up to shrinking the size of the dense open $\Omega$ (and, therefore, of $M$ and $M'$), the composition $\beta'=\beta \circ \sigma$ is simple (with respect to $E'$).
\end{lemma}

\begin{proof}
	It is enough to prove the claim in the case that $\sigma$ is a single blowup with center $\mathcal{C}$. 
	By Remmert's Proper Mapping Theorem applied to $\beta|_E$, the image $\beta(\mathcal{C})$ is a closed analytic subset of $\Omega$. 
	If $\mbox{dim}\big(\beta(\cal C)\big) <  \mbox{dim}(\Omega)$, set $Z=\beta(\cal C)$ and note that, once we replace $\Omega$ by $\Omega \setminus Z$, the result easily follows from the fact that $\sigma\colon M' \to M$ is an isomorphism. 
	We can therefore assume that $\mbox{dim}\big(\beta(\mathcal{C})\big) = \mbox{dim}(\Omega)$, so that $\beta(\mathcal{C})=\Omega$. 
	Since $\mbox{dim}(\mathcal{C}) \leq \mbox{dim}(\Omega)$ by hypothesis, we conclude that $\mbox{dim}(\mathcal{C}) = \mbox{dim}(\Omega)$, and in particular the restriction $\beta|_{\mathcal{C}}\colon\mathcal{C} \to \Omega$ is generically a local isomorphism. 
	Let $Y \subset \mathcal{C}$ be the set of critical points of $\beta|_{\mathcal{C}}$, which is a proper closed analytic subset of $\mathcal{C}$. 
	Again by Remmert's Proper Mapping Theorem, the image $Z'=\beta(Y)$ is a closed analytic subset of $\Gr(n-2,\C^n)$, properly contained in $\Gr(n-2,\C^n)$ because $\mbox{dim}(Y)<\mbox{dim}\big(\Gr(n-2,\C^n)\big)$. 
	Now, after replacing $\Omega$ by $\Omega \setminus Z'$, we can assume that $\beta \colon \mathcal{C} \to \Omega$ is everywhere a local isomorphism. 
	We now claim that $\beta'$ is simple via direct computation.
	Indeed, since smoothness can be verified locally, let us fix a point $p \in \mathcal{C}$, and denote by $f_1$ and $f_2 \in \mathcal{O}_{p}$ local generators of $\mathcal{C}$. 
	Since $\beta$ is simple at $p$, there exists an (analytic) local coordinate system $(\lambda,x_1,x_2)$ at $p$ such that $\beta(\lambda,x_1,x_2)=\lambda$ and $E$ is locally contained in $(x_1x_2=0)$. 
	Since $\pi\colon\mathcal{C}\to \Omega$ is a local isomorphism around $p$ and $\mathcal{C}$ is smooth and adapted to $E$, if follows that the map $(\lambda,x_1,x_2) \to (\lambda,f_1,f_2)$ is a local isomorphism and $E \subset (f_1f_2=0)$. 
	Therefore, up to a local change of variables, we can assume that $f_1=x_1$ and $f_2=x_2$, and we easily conclude that $\beta' \colon M' \to \Omega$ is simple.
\end{proof}

Now, recall that we have an embedding of $(X,0)$ in $(\C^n,0)$ and let $\Phi \colon (X,0) \times \Omega \to (\C^2,0) \times \Omega$ be the morphism defined by $\Phi (x,\cal D) = \big(\ell_{\cal D}(x), \cal D\big)$, which is generically of maximal rank. 
Let $\pi\colon (X_{\pi},E) \to (X,0)$ be a good resolution of $(X,0)$ which factors through the blowup of its maximal ideal and through its Nash transform. 
We note that, by using \cite[Lemma 5.2]{Laufer1971} (a special case of the direct image theorem of Grauert), resolution of singularities, and the universal property of blowups, there exists a sequence of blowups $\alpha\colon (\mathcal{Z},\cal G) \to (X_{\pi},E) \times \Omega$ and an analytic morphism $\Psi\colon (\mathcal{Z},\cal G) \to (\mathcal{Y},\cal F)$ such that $\Psi^{-1}(\cal F)_{\mathrm{red}} = \cal G_{\mathrm{red}}$ and the following diagram
\begin{equation}\label{eq:CommutativeDiagram}
\xymatrix@C=3.5em{ (\mathcal{Z},\cal G) \ar[r]^(.41){\alpha} \ar[d]^{{\Psi}}  &(X_{\pi},E) \times \Omega \ar[r]^(.52){\pi \times \mathrm{Id}}      &  (X,0) \times \Omega \ar[d]^{\Phi} & \\
	 (\mathcal{Y},\cal F) \ar[rr]^(.48){\sigma} & &(\C^2,0) \times \Omega \ar[r]^(.63){\beta} & \Omega}
\end{equation}
is commutative, with $\beta_{\cal Y}=\beta\circ\sigma$ simple. 
Thanks to Lemma \ref{cl:TechBlowingup}, up to shrinking the size of the open $\Omega$ if necessary, the morphism $\beta_{\cal Z}=\beta_{\cal Y}\circ \Psi$ is simple as well. 
We are now ready to complete the proof of Lemma \ref{lem:generic projection}.

\begin{proof}[Proof of Lemma \ref{lem:generic projection}]
	The map $\widetilde{\ell }|_{\Gamma_{\pi}}$ is determined by its restriction to the set of divisorial points of $\Gamma_{\pi}$, as those form a dense subset of $\Gamma_{\pi}$. 
	Since $\beta_{\cal Z}$ and $\beta_{\cal Y}$ are simple, for every pair of elements $\cal D, \cal D'$ of $\Omega$, the following diagram commutes
	\[
	\xymatrix{    
		&V(\Gamma_{\pi})    \ar[ld]_{\widetilde\ell_{\cal D} }  \ar[rd]^{\widetilde\ell_{\cal D'}}         &  \\
		V(\Gamma_{\sigma_{\cal D}})   \ar[rr]_{\widetilde{\rho}_{\cal D, \cal D'}}  & &  V( \Gamma_{\sigma_{\cal D'}})
	}
	\]
	We now need to prove the result on the divisorial points of $\Gamma_{\pi}$ which are not vertices of $\Gamma_{\pi}$. 
	It is sufficient to consider the case where $v$ is the divisorial point associated with the exceptional curve of the blowup $\pi'\colon (X_{\pi}',E')\to (X_{\pi},E)$ of center $E_{v_1} \cap E_{v_2}$, since the same argument can then be repeated verbatim for general sequence of point blowups. 
	Observe that if $E_v \times \Omega$ is already a component of $\cal G$, then $\widetilde{\ell}_{\cal D} (v) \in \Gamma_{\sigma_{\cal D}}$ for every ${\cal D} \in \Omega$, and we conclude easily. 
	If $\cal G_v = E_v \times \Omega$ is not a component of $\cal G$, we note that $\cal G_{v_1}\cap\cal G_{v_2} = (E_{v_1} \cap E_{v_2}) \times \Omega$ is an admissible center in $(\mathcal{Z},\cal G)$, since all blowups in $\alpha$ are admissible. 
	We therefore may perform this extra blowup $\alpha'\colon (\mathcal{Z}',\cal G') \to (\mathcal{Z},\cal G)$, whose exceptional divisor $\cal G_v = E_v \times \Omega$ is trivial with respect to the family structure. 
	Fix ${\cal D} \in \Omega$, set $w_1 = \ell_{\cal D}(v_1)$ and $w_2 = \ell_{\cal D}(v_2)$, and consider the associated components $\cal F_{w_1}^{\cal D}$ and $\cal F_{w_2}^{\cal D}$ of $\cal F$.
	Then, after performing a sequence of combinatorial blowups  $\rho\colon(Y',\cal F') \to (Y,\cal F)$, starting with blowing up the center $\cal F_{w_1}^{\cal D} \cap \cal F_{w_2}^{\cal D}$, the projection $\widetilde{\ell}_{\cal D'} (v) $ belongs to the graph of $  \Gamma_{\rho_{{\cal D'}} \circ \sigma_{\cal D'}}$ for every $\cal D'$ in $\Omega$.
	We have obtained, without the need to shrink the size of $\Omega$, the following commutative diagram:
	\[
	\xymatrix@C=3.5em{ (\mathcal{Z}',\cal G') \ar[d]_{\Psi'} \ar[r]^{\alpha'}  &
		(\mathcal{Z},\cal G)  \ar[d]^{\Psi} \ar[r]^(.4){\alpha} & (X_{\pi},E) \times \Omega \ar[r]^(.52){\pi \times \mathrm{Id}}      &  (X,0) \times \Omega \ar[d]^{\Phi} & \\
		(\mathcal{Y}',\cal F')  \ar[r]^{\rho}  &  (\mathcal{Y},\cal F) \ar[rr]^(.48){\sigma}&& (\C^2,0) \times \Omega \ar[r]^(.63){\beta} & \Omega}
	\]
	where $\beta_{\mathcal{Y}'}=\beta\circ\sigma\circ \rho$ and  $\beta_{\mathcal{Z}'}=\beta_{\mathcal{Y}'}\circ\Psi'$ are simple morphisms. 
	We conclude easily.
\end{proof}

\begin{remark}
	If $(X,0)$ is a hypersurface in $(\C^3,0)$, shrinking the open set $\Omega$ is not necessary when applying Lemma~\ref{cl:TechBlowingup}, since a resolution of the family can be constructed everywhere by performing a Hirzebruch--Jung process in family, exploiting the fact that, thanks to \cite[Corollary~3.4]{PopescuPampu2002} (or, more generally, to \cite[Theorem~5.1]{PopescuPampu2004}), the combinatorial data of the quasi-ordinary singularities that appear during the process are constant in the family.
\end{remark}

We conclude the section by recalling the definition of the local degree of a divisorial point $v$ of $\NL(X,0)$, as it will be very important in the remaining part of the paper.
Let $\ell\colon (X,0)\to(\C^2,0)$ be a generic projection of $(X,0)$ and consider the diagram~\eqref{eqn:diagram_defining_local_degree} (page~\pageref{eqn:diagram_defining_local_degree}).
For each component $E_{\nu}$ of $\pi_{\ell,v}^{-1}(0)$ (respectively $E_{\nu'}$ of  $\sigma_{\ell,v}^{-1}(0)$), let us choose a tubular neighborhood disc bundle $N(E_{\nu})$ (resp. $N(E_{\nu'})$), and consider the two sets
\begin{equation*}\label{notation:calN(E)}
\cal N(E_{v})=N(E_v)\setminus \bigcup_{E_\nu\neq E_v}N(E_\nu) 
\,\,\,\,\,\,\,\,\mbox{and}\,\,\,\,\,\,\,\,
\cal N(E_{\widetilde\ell(v)})=N(E_{\widetilde\ell(v)})\setminus \bigcup_{E_{\nu'}\neq E_{\widetilde\ell(v)}}N(E_{\nu'})
\end{equation*}
in $X_{\pi_{\ell,v}}$ and $Y_{\sigma_{\ell,v}}$ respectively.
We can then adjust the disc bundles $N(E_{\nu})$  and $N(E_{\nu'})$  in such a way that the cover $\ell$ restricts to a cover  
\begin{equation}\label{eqn:cover_local_degree}
\ell_{v} \colon \pi_{\ell,v}\big(\cal N(E_{v})\big) \longrightarrow \sigma_{\ell,v}\big(\cal N(E_{\widetilde\ell(v)})\big)
\end{equation}
branched precisely on the polar curve of $\ell$ (if $v$ is not a $\cal P$-node, the branching locus is just the origin).
Using a resolution in family over $\Omega$ as in the proof of the Lemma \ref{lem:generic projection}, it is easy to deduce the following result.

\begin{lemma}\label{lem:degre}
For every divisorial point $v$ of $\NL(X,0)$, the degree $\deg(\ell_{v})$ of the cover $\ell_v$ does not depend on the choice of a generic projection $\ell\colon (X,0) \to (\C^2,0)$.
\end{lemma}

Therefore, we can set $\deg(v) = \deg(\ell_{v})$.
We call this integer the \emph{local degree of a generic projection of} $(X,0)$ \emph{at} $v$, or simply the \emph{local degree of $(X,0)$ at} $v$. Note that if $\deg(v) = 1$, then the map \eqref{eqn:cover_local_degree} is an isomorphism.


\section{Generic projections of LNE surfaces}
\label{sec:key_lemma_multiplicities}

In this section we study LNE surface germs by establishing some properties related to their generic projections.

\medskip 

We begin by proving the invariance of multiplicities under generic projections, and showing a characterization of the $\cal P$-nodes of a LNE normal surface in terms of their local degrees.
More precisely, we prove the following result:

\begin{lemma} \label{lem:projections_LNE}  
	Let $(X,0)$ be a LNE normal surface germ, let $\ell\colon (X,0)\to (\C^2,0)$ be a generic projection, let $\pi \colon X_{\pi} \to X$ be the minimal good resolution of $(X,0)$ which factors through its Nash transform, and let $v$ be a divisorial point of $\Gamma_{\pi} \subset \NL(X,0)$.
	Then:
	\begin{enumerate}
		\item \label{lem:projections_LNE_multiplicities} 
		$m_v = m_{\widetilde{\ell}(v)}$\,;
		\item \label{lem:projections_LNE_degree}
		$v$ is a $\cal P$-node of $(X,0)$ if and only if  $\deg{v}>1$.
	\end{enumerate}		
\end{lemma}

	Before delving into the proof of the lemma, let us recall the notions of inner and outer contact and that of inner rate.
	Let $(\gamma,0)$ and $(\gamma',0)$ be two distinct real or complex curve germs on the surface germ $(X,0)\subset(\C^n,0)$ and denote $S_{\epsilon}$ the sphere in $\C^n$ having center $0$ and radius $\epsilon>0$.
	The \emph{inner contact} between $\gamma$ and $\gamma'$ is the rational number $q_{\inn}=q_{\inn}(\gamma, \gamma')$ defined by 
	\[
	d_{\inn} \big(\gamma \cap S_{\epsilon}, \gamma' \cap S_{\epsilon}\big) = \Theta(\epsilon^{q_{\inn}}),
	\]
	where $\Theta$ stands for the big-Theta asymptotic notation of Bachmann--Landau, which is defined as follows: given two function germs $f,g\colon \big([0,\infty),0\big)\to \big([0,\infty),0\big)$ we say that $f$ is \emph{big-Theta} of $g$, and we write $f(t) = \Theta \big(g(t)\big)$, if there exist real numbers $\eta>0$ and $K >0$ such that ${K^{-1}}g(t) \leq f(t) \leq K g(t)$ for all $t\geq0$ satisfying $f(t)\leq \eta$.  
	The \emph{outer contact} $q_{\out}(\gamma, \gamma')$ is defined in an analogous way, by using the outer metric $d_{\out}$ instead of the inner metric $d_{\inn}$.
	Observe that if $(X,0)$ is LNE then $q_{\inn}(\gamma, \gamma') = q_{\out}(\gamma, \gamma')$.
	Recall that the \emph{inner rate} $q_v$ of a divisorial point $v$ of $\NL(X,0)$ is defined as the inner contact $q_\mathrm{inn}(\gamma,\gamma')$, where $\gamma,\gamma'\subset (X,0)$ are two curve germs that pullback to two curvettes through distinct points of the divisor $E_v$ associated with $v$ via any good resolution $\pi\colon X_\pi \to X$ of $(X,0)$ that makes the divisor $E_v$ appear.
	This definition only depends on the divisorial point $v$ (see \cite[Lemma~3.2]{BelottodaSilvaFantiniPichon2019}).

\begin{proof}
	We begin by proving \ref{lem:projections_LNE_multiplicities}. 
	Write $\ell=\ell_{\cal D}$ and set $w = \widetilde{\ell}_{\cal D}(v)$.
	Consider maps
	\begin{equation*}
		\xymatrix@C=5em{    
			X_{\pi_{\ell,\cal D}}         \ar[d]_{\widehat\ell_\cal{D}}    \ar[r]^{\pi_{\ell,\cal D}} & X    \ar[d]^{\ell_\cal D}           \\
			Y_{\sigma_{\ell,v\cal D}}      \ar[r]_{\sigma_{\ell,\cal D}}  &  \C^2
		}
	\end{equation*}
	as in Diagram~\ref{eqn:diagram_defining_local_degree} (page~\pageref{eqn:diagram_defining_local_degree}) such that $\pi_{v,\cal D}$ is a good resolution of $(X,0)$ factoring through its Nash transform (and therefore through $\pi$), $\sigma_{v,\cal D}$ is a sequence of point blowups of $(\C^2,0)$ such that $w$ is associated with a component $E_w$ of $(\sigma_{v,\cal D})^{-1}(0)$, and the component $E_v$ of $(\pi_{v,\cal D})^{-1}(0)$ associated with $v$ is sent by $\widehat \ell$ surjectively onto $E_w$.
	
	Take a curvette $\gamma^*$ of $E_w$ which does not intersect a component of the strict transform of  the discriminant curve $\Delta_{\cal D}$ of $\ell_\cal D$ and let $(\gamma,0)\subset(\C^2,0)$ be the irreducible curve germ defined by $\gamma=\sigma_{v,\cal D}(\gamma^*)$, so that we have $m_w=\mult(\gamma)$.
	Up to replacing $\gamma^*$ by a nearby curvette, among the components of $({\ell}_{\cal D})^{-1}(\gamma)$ we can find an irreducible curve germ $\widehat\gamma$ on $(X,0)$ whose strict transform by $\pi_{v,\cal D}$ is a curvette of $E_v$, so that we have $m_v=\mult(\widehat\gamma)$. 
	We then have $\mult(\widehat{\gamma}) = k \mult(\gamma)$, where $k$ is the degree of the covering $\widehat{\gamma}\to\gamma$ induced by $\ell$.
	
	We will argue by contradiction. 
	Assume that $\mult(\widehat\gamma)\neq \mult(\gamma)$, that is that $k>1$.  
	Our goal will be to construct two real arcs $\widehat\delta_1$ and $\widehat\delta_2$ inside $\widehat\gamma$ whose inner and outer contacts do not coincide; this will then imply that $(X,0)$ is not LNE, contradicting our hypothesis.	
	In order to do so, we consider another generic projection $\ell_{\cal D'} \colon (X,0) \to (\mathbb C^2,0)$, chosen to be generic with respect to the curve $\widehat{\gamma}$ as well, and set $\gamma'=\ell_{\cal D'}(\widehat{\gamma})$. 
	Then the cover $\widehat{\gamma}\to\gamma'$ induced by $\ell_{\cal D'}$ has degree 1, and thus $\widehat\gamma$ and $\gamma'$ have the same multiplicity since $\mult(\widehat{\gamma}) = \mbox{degree}(\ell_{\cal D'}|_{\widehat{\gamma}})\mult(\gamma')= \mult(\gamma')$.	
	Set $w' = \widetilde{\ell}_{\cal D'}(v)$.
	By Lemma~\ref{lem:generic projection}, we have $w'  = \eta_{\cal D, \cal D'}(w)$, 
%
%
	and thus $m_{w'} = m_w$ and $q_{w'} = q_v = q_w$. 
	Moreover, by definition of $\eta_{\cal D, \cal D'}$, the strict transform of $\gamma'$ by $\sigma_{v,\cal D'}$ intersects $(\sigma_{v,\cal D'})^{-1}(0)$ in a smooth point $p$ of $E_{w'}$.
	
	Observe that, since the plane curve germ $\gamma$ is the image through $\sigma_{v,\cal D}$ of a curvette of $E_w$, it has no characteristic Puiseux exponent strictly greater than the inner rate $q_w$ of $w$.
	On the other hand, the strict transform of $\gamma'$ by   $\sigma_{v,\cal D'}$ cannot be a curvette of $E_{w'}$ since  $\mult(\gamma') = \mult(\widehat\gamma) = k m_{w'}>m_{w'}$.
	Therefore, the minimal good embedded resolution of $\gamma'$ is obtained by composing $\sigma_{v,\cal D'}$ with a nontrivial sequence of point blowups, starting with the blowup of $Y_{\sigma_{v,\cal D'}}$ at $p$. 
	Let $E_{w''}$ be the last irreducible curve created by this sequence, so that the strict transform of $\gamma'$ is a curvette of  $E_{w''}$. 
	Then the inner rate $q_{w''}$ of $E_{w''}$, which is strictly greater than $q_{w'}=q_w$, is a characteristic Puiseux exponent of $\gamma'$.
	
	Let us choose an embedding $(X,0) \subset (\C^n,0)$ and coordinates  $(x_1,\ldots,x_n)$ of $\C^n$ such that $\ell_{\cal D'}(x) = (x_1,x_2)$ and $\gamma'$ is not tangent to the line $x_1=0$.
	Then, since $q_{w''}$ is a characteristic Puiseux exponent of $\gamma'$, we can find a pair of real arcs $\delta_1'$ and $\delta_2'$ among the components of the intersection $\gamma' \cap \{x_1=t \,|\, t \in \R\}$ such that their contact $q(\delta_1',\delta_2')$ is equal to $q_{w''}$ (we refer to \cite[\S3]{NeumannPichon2014} for details on this classical result about Puiseux expansions).
	Let $\widehat\delta_1$ and $\widehat\delta_2$ be two liftings of $\delta_1'$ and $\delta_2'$ via $\ell'$.
	Since the projection $\ell_{\cal D'}$ is generic with respect to $\widehat\gamma$, it induces by \cite[pp. 352-354]{Teissier1982} a bilipschitz homeomorphism for the outer metric from $\widehat{\gamma}$ onto $\gamma'$, and therefore the outer contacts $q_{\out}(\widehat{\delta}_1,\widehat{\delta}_2)$ and $q(\delta_1',\delta_2')$ coincide, so that in particular we have
	\begin{equation}\label{eqn:outer_contact_curvettes}
	q_{\out}(\widehat{\delta}_1,\widehat{\delta}_2)=q_{w''}>q_v.
	\end{equation}
		
	We will now show that the inner contact $q_{\inn}(\widehat{\delta}_1,\widehat{\delta}_2)$ between $\widehat{\delta}_1$ and $\widehat{\delta}_2$ is at most $q_v$, which will yield the contradiction we were after.
	Observe that the inner contact $q = q_{\inn}^X(\widehat{\delta}_1,\widehat{\delta}_2)$ between $\widehat{\delta}_1$ and $\widehat{\delta}_2$ can also be computed as $d^{F_t}_{\inn}\big(\widehat{\delta}_1(t),\widehat{\delta}_2(t)\big) = \Theta(t^q)$, where $d^{F_t}_{\inn}\big(\widehat{\delta}_1(t),\widehat{\delta}_2(t)\big)$ denotes the inner distance between $\widehat{\delta}_1(t)$ and $\widehat{\delta}_2(t)$ inside the Milnor fiber  $F_t = X \cap \{x_1=t\}$, that is the distance measured by taking the infimum of the inner lengths of the paths joining $ \widehat{\delta}_1(t)$ to $\widehat{\delta}_2(t)$ inside $F_t$.
	This is a consequence of the fact that, by \cite{BirbrairNeumannPichon2014} and in the language therein, the subset $\pi\big({\cal N}(E_v)\big)$ of $(X,0)$ is a $B(q_v)$-piece fibered by the restriction of the  generic linear form $x_1$ whenever $q_v>1$, while it is a conical piece if $q_v=1$. 
	
	In order to conclude, consider a small disc $D$ contained in the divisor $E_v$ and centered at the point $\widehat{\gamma}^*\cap E_v$ and let $N \cong D \times D'$ be a trivialization of the normal disc-bundle to $E_v$ over $D$ such that $\widehat{\gamma}^* = \{0\} \times D'$. 
	The intersection $F_t \cap \pi(N)$ consists of $m_v$ disjoint discs each centered at one of the $m_v$ distinct points of $\widehat{\gamma} \cap F_t$.
	Since $\delta_1(t)$ and $\delta_2(t)$ are two of these points, then they are the centers of two of these discs, $D_1$ and $D_2$ respectively.  
	Since these two discs have diameters $\Theta(t^{q_v})$, any path from $\widehat{\delta}_1(t)$ to $\widehat{\delta}_2(t)$ inside $F_t$ will have intersections with $D_1$ and $D_2$ of length at least $\Theta(t^{q_v})$.
	Therefore $q_{\inn}( \widehat{\delta}_1,\widehat{\delta}_2) \leq q_v$, and so $q_{\inn}( \widehat{\delta}_1,\widehat{\delta}_2) <q_{\out}( \widehat{\delta}_1,\widehat{\delta}_2)$, which contradicts the fact that $(X,0)$ is LNE.
	This completes the proof that $m_v=m_{\widetilde\ell_{\cal D}(v)}$.	
	
	Let us now prove \ref{lem:projections_LNE_degree}.
	If $v$ is a $\cal P$-node, then it immediately follows from the definition of degree that $\deg(v)>1$, because the cover $\ell$ is ramified in a neighborhood of the polar curve.
	Assume that $v$ is not a $\cal P$-node and that $\deg{v}>1$.
	We use again the plane curve $\gamma = \sigma_{v,\cal D}(\gamma*)$ introduced in the proof of \ref{lem:projections_LNE_multiplicities}. 
	
	By the definition of $\deg(v)$, the curve $\ell^{-1}(\gamma)$ has $k_v$ irreducible components whose strict transforms by $\pi_{v, \cal D}$ are curvettes of $E_v$, where $k_v$ divides $\deg(v)$, and we have $m_{v} = m_{\widetilde{\ell}(v)}{\deg(v)}/{k_v}$.
	Since $m_{v} =  m_{\widetilde{\ell}(v)}$ by \ref{lem:projections_LNE_multiplicities}, then $\deg(v) = k_v$, so $k_v>1$.  
	Let $\widehat{\gamma}_1$ and $\widehat{\gamma}_2$ be two components of $\ell^{-1}(\gamma)$ whose strict transforms by $\pi_{v, \cal D}$ are curvettes of $E_v$ (as was the case in part~\ref{lem:projections_LNE_multiplicities}, two such components can always be found after replacing $\gamma^*$ by a nearby curvette if necessary), and let us consider two real arcs  $\widehat{\delta}_1 \subset  \widehat{\gamma}_1$ and   $\widehat{\delta}_2 \subset  \widehat{\gamma}_2$ such that $\ell_{\cal D}(\widehat{\delta}_1 ) = \ell_{\cal D}(\widehat{\delta}_2)$.
	By definition of $q_v$, we have $q_{\inn}(\widehat{\gamma}_1, \widehat{\gamma}_2) = q_v$ and then $q_{\inn}(\widehat{\delta}_1, \widehat{\delta}_2) = q_v$. 
	Since $v$  is not a $\cal P$-node, the lifted Gauss map ${\lambda}$ on  $X_\pi$ (see \cite[Definition 6.11]{NeumannPedersenPichon2020a}) is constant along $E_v$ and we then have $\lambda(p_1) = \lambda(p_2)$, see page 19 in \emph{loc.\ cit}. 
	By \cite[Lemma 9.1]{NeumannPedersenPichon2020a}, this implies that $q_{\inn}(\delta_1, \delta_2) < q_{\out}(\delta_1, \delta_2)$, therefore  $(X,0)$ is not LNE, contradicting the hypothesis.
\end{proof}

\begin{remark}\label{remark:multiplicities_for_P-nodes}
Whenever $v$ is not a $\cal P$-node, then part \ref{lem:projections_LNE_multiplicities} of the lemma is an immediate consequence of \ref{lem:projections_LNE_degree} (whose proof is more elementary and independent on the proof of \ref{lem:projections_LNE_multiplicities}). 
Indeed, consider the degree $\deg(v)$ cover 
$\ell_{v} \colon \pi_{\ell,v}\big(\cal N(E_{v})\big) \to \sigma_{\ell,v}\big(\cal N(E_{\widetilde\ell(v)})\big)$
of equation~\eqref{eqn:cover_local_degree} (page \pageref{eqn:cover_local_degree}), and choose coordinates of $\C^2$ so that $\ell=(z_1,z_2)$, with $h=z_1$ a generic linear form on $(X,0)$.
Then $\ell_v$ restricts to a degree $\deg(v)$ cover from the intersection $F_v = \cal N(E_{v})\cap \{h=t\}$ to its image $\ell(F_v)$, implying that $m_v = k m_{\widetilde{\ell}(v)}$, where the integer $k$ divides $\deg(v)$.
\end{remark}

As a simple consequence of Lemma~\ref{lem:projections_LNE}.\ref{lem:projections_LNE_degree} we deduce the following result. 

\begin{corollary}\label{cor:genus_P-node}
Let $(X,0)$ be a LNE normal surface germ, let $v$ be a divisorial point of $\NL(X,0)$, and assume that $v$ is associated with a genus $g>0$ component $E_v$ of the exceptional divisor of some good resolution of $(X,0)$.
Then $v$ is a $\cal P$-node of $(X,0)$.	
\end{corollary}

\begin{proof}
Consider again the finite cover 
$\ell_{v} \colon \pi_{\ell,v}\big(\cal N(E_{v})\big) \to \sigma_{\ell,v}\big(\cal N(E_{\widetilde\ell(v)})\big)$
of equation~\eqref{eqn:cover_local_degree} (page \pageref{eqn:cover_local_degree}), and assume that $v$ is not a $\cal P$-node. 
Then $\ell_{v}$ is a homeomorphism by Lemma~\ref{lem:projections_LNE}.\ref{lem:projections_LNE_degree}, and so is its restriction $\ell_{v}|_{\cal N(E_{v}) \cap E_v} \colon   \cal N(E_{v}) \cap E_v \to   \cal N(E_{\widetilde\ell(v)}) \cap E_{\widetilde\ell(v)}$.  
Observe that $\cal N(E_{v}) \cap E_v$ (respectively $N(E_{\widetilde\ell(v)}) \cap E_{\widetilde\ell(v)}$) is the complex curve $E_v$ (resp. $E_{\widetilde\ell(v)}$) with a finite union of discs removed.
Since $E_{\widetilde\ell(v)}$ has genus zero, this implies that $E_v$ also has genus zero. 
\end{proof}

\begin{example}\label{ex:pi_not_minimal_multiplicities}
Let us show with an example that the minimality of $\pi$ is a necessary hypothesis in Lemma~\ref{lem:projections_LNE}.
Let $(X,0)$ be the standard singularity $A_2$, which is the hypersurface singularity in $(\C^3,0)$ defined by the equation $x^2+y^2+z^3=0$.
A good resolution $\pi\colon X_\pi\to X$ of $(X,0)$ can be obtained by the method described in \cite[Chapter II]{Laufer1971}.
It considers the generic projection $\ell = \ell_{\cal D} = (y,z) \colon (X,0) \to (\C^2,0)$ and, given a suitable embedded resolution $\sigma_{\Delta} \colon Y_{\sigma_{\Delta}} \to \C^2$ of the associated discriminant curve $\Delta \colon y^2+z^3=0$, gives a simple algorithm to compute a resolution of $(X,0)$ as a cover of $Y$.
In this example, $\Delta$ is a cusp and the dual graph $\Gamma_{
	\sigma_\Delta}$ of its minimal embedded resolution $\sigma_\Delta$ is depicted on the left of Figure~\ref{figure:example_graphs}.
Its vertices are labeled as $w_0, w_1$, and $w_2$ in their order of appearance as exceptional divisors of point blowups in the resolution process, the negative number attached to each vertex denotes the self-intersection of the corresponding exceptional curve, the positive numbers in parentheses denote the multiplicities, and the arrow denotes the strict transform of $\Delta$. 
In this case Laufer's method gives us the dual graph of a good resolution $\pi_\ell$ of $(X,0)$ such that $\ell\circ \pi_\ell$ factors through $\sigma_\Delta$, appearing as the graph in the middle of Figure~\ref{figure:example_graphs}.
Again, all exceptional components are rational, each vertex is decorated by the self-intersection of the corresponding exceptional curve and with its multiplicity, the arrow denotes the strict transform of $\Delta$, and the vertices are labeled in a way that $\widetilde\ell(v_0)=\widetilde\ell(v'_0)=w_0$, $\widetilde\ell(v_1)=w_1$, and $\widetilde\ell(v_2)=w_2$.
Observe that the vertex $v_1$ has multiplicity 2, but it is sent by $\widetilde\ell$ to the vertex $w_1$, which has multiplicity 1.
However, the rational curve $E_{v_1}$ associated with the vertex $v_1$ has self-intersection $-1$ and can thus be contracted.
The resulting map $\pi\colon X_\pi\to X$, which no longer factors through $\sigma_\Delta$, is the minimal resolution of $(X,0)$ factoring through its Nash transform.
Observe that its $\cal P$-node $v_2$ can also be contracted, yielding the minimal good resolution of $(X,0)$, which in this case does not factor through the Nash transform of $(X,0)$.
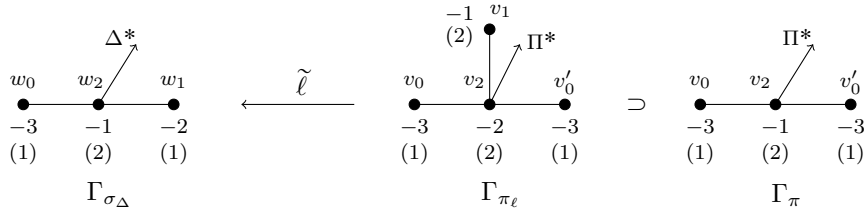
\begin{figure}[h] 
	\centering
	\begin{tikzpicture}
	\node(a)at(-.85,-1.2){$\Gamma_{\sigma_\Delta}$};
	
	\draw[thin ](-2,0)--(0,0);
	\draw[fill ] (-2,0)circle(2pt);
	\draw[fill ] (-1,0)circle(2pt);
	\draw[fill ] (0,0)circle(2pt);
	
	\draw[thin,>-stealth,->](-1,0)--+(0.5,0.8);

	\begin{footnotesize}		
	\node(a)at(-2,-0.3){$-3$};
	\node(a)at(-1,-0.3){$-1$};
	\node(a)at(0,-0.3){$-2$};
	
	\node(a)at(-2,-0.65){$(1)$};
	\node(a)at(-1,-0.65){$(2)$};
	\node(a)at(0,-0.65){$(1)$};
	
	\node(a)at(-2,0.3){$w_0$};
	\node(a)at(-1.1,0.3){$w_2$};
	\node(a)at(0,0.3){$w_1$};
	\node(a)at(-.7,0.9){$\Delta^*$};
	\end{footnotesize}

	\begin{scope}[xshift=4.2cm]
	
	\draw[>-stealth,->](-1.8,0)--+(-1.5,0);
	\node(a)at(-2.5,0.3){$\widetilde\ell$};
	
	\draw[thin ](-1,0)--(1,0);
	\draw[thin ](0,0)--(0,1);
	\draw[fill ] (-1,0)circle(2pt);
	\draw[fill ] (0,0)circle(2pt);
	\draw[fill ] (1,0)circle(2pt);		
	\draw[fill ] (0,1)circle(2pt);
	
	\draw[thin,>-stealth,->](0,0)--+(+0.4,0.8);
	
	\begin{footnotesize}
	\node(a)at(-1,0.3){$v_0$};
	\node(a)at(1,0.3){$v_0'$};
	\node(a)at(-.2,0.3){$v_2$};
	\node(a)at(.15,1.25){$v_1$};
	
	\node(a)at(.7,.85){$\Pi^*$};
	
	\node(a)at(-1,-0.3){$-3$};
	\node(a)at(0,-0.3){$-2$};
	\node(a)at(1,-0.3){$-3$};
	\node(a)at(-0.4,1.2){$-1$};
	
	\node(a)at(-1,-0.65){$(1)$};
	\node(a)at(0,-0.65){$(2)$};
	\node(a)at(1,-0.65){$(1)$};
	\node(a)at(-.4,.9){$(2)$};
	
	\end{footnotesize}
	
	\node(a)at(0.15,-1.2){$\Gamma_{\pi_\ell}$};
	
	\end{scope}

	\begin{scope}[xshift=8cm]
	
	\node(a)at(-1.85,0){$\supset$};
	
	\draw[thin ](-1,0)--(1,0);
	\draw[fill ] (-1,0)circle(2pt);
	\draw[fill ] (0,0)circle(2pt);
	\draw[fill ] (1,0)circle(2pt);
	
	\draw[thin,>-stealth,->](0,0)--+(0.5,0.8);
	
	\begin{footnotesize}
	\node(a)at(.3,0.9){$\Pi^*$};
	
	\node(a)at(-1,0.3){$v_0$};
	\node(a)at(1,0.3){$v_0'$};
	\node(a)at(-.2,0.3){$v_2$};
	
	\node(a)at(-1,-0.3){$-3$};
	\node(a)at(0,-0.3){$-1$};
	\node(a)at(1,-0.3){$-3$};
	
	\node(a)at(-1,-0.65){$(1)$};
	\node(a)at(0,-0.65){$(2)$};
	\node(a)at(1,-0.65){$(1)$};
	
	\end{footnotesize}
	
	\node(a)at(0.15,-1.2){$\Gamma_{\pi}$};
	
	\end{scope}
	
	%
	%
	%
	%
	%
	%
	%
	\end{tikzpicture}
	\caption{Dual resolution graphs for the plane curve $\Delta$ (left) and for the surface singularity $X=A_2$ (middle and right).}
	\label{figure:example_graphs}
\end{figure}	   

\noindent In the proof of Lemma~\ref{lem:projections_LNE}, the minimality of $\pi$ is only required in order to apply Lemma~\ref{lem:generic projection}.
Therefore, this examples also shows how the commutativity of the diagram of Lemma~\ref{lem:generic projection} may fail to hold on a larger dual graph  such as $\Gamma_{\pi_\ell}$.
\end{example}

We can now move our focus to the morphism $\widetilde\ell$ induced by a generic projection $\ell\colon (X,0)\to(\C^2,0)$, and more precisely to its restriction to the dual graph $\Gamma_\pi$ of some good resolution of $(X,0)$.
Recall that, given a graph $\Gamma$, we denote by $V(\Gamma)$ the set of its vertices.
In general, even whenever $\pi$ factors through the Nash transform of $(X,0)$, it is not possible to find a suitable sequence of point blowups $\sigma\colon Y \to \C^2$ of $(\C^2,0)$ such that $\widetilde\ell$ induces a morphism of graphs $\widetilde\ell|_{\Gamma_\pi}\to\Gamma_\sigma$, since in order to make the elements of $\widetilde\ell\big(V(\Gamma_\pi)\big)$ appear among the vertices of $\Gamma_\sigma$, one usually introduces too many additional vertices, so that the image $\widetilde\ell(e)$ of some edge $e$ of $\Gamma_\pi$ is not an edge of $\Gamma_\sigma$, but only a string of several edges.
Remarkably, thanks to Lemma~\ref{lem:projections_LNE}.\ref{lem:projections_LNE_degree}, in the case of LNE surfaces we can control this phenomenon completely.
Indeed, the following proposition explains that in this case we do get a morphism of graphs, provided that we restrict our attention to a subgraph of $\Gamma_\pi$ that does not contain a $\cal P$-node of $(X,0)$ in its interior.

\begin{proposition}\label{prop:morphism_graph_local}
	Let $(X,0)$ be a LNE normal surface germ, let $\pi\colon X_\pi\to X$ be the minimal good resolution of $(X,0)$ that factors through its Nash transform, let $\ell\colon (X,0)\to (\C^2,0)$ be a generic projection, and let $\widetilde\ell\colon \NL(X,0)\to\NL(\C^2,0)$ be the map induced by $\ell$.
	Let $S$ be a subset of $V(\Gamma_\pi)$  which contains all $\cal{P}$-nodes. Let $W$ be one of the connected components of $\Gamma_{\pi}\setminus S$, and let $\Gamma_0$ be the subgraph of $\Gamma_\pi$ whose underlying topological space is the closure of $W$ in $\Gamma_\pi$. Let $\sigma_{\Gamma_0} \colon Y_{\sigma_{\Gamma_0}} \to \C^2$ be the minimal sequence of point blowups such that $\widetilde{\ell}\big(V (\Gamma_0)\big) \subset V(\Gamma_{\sigma_{\Gamma_0}})$. 
	Then
	\begin{enumerate}
\item  \label{prop:morphism_graph_local_1} 
	$\sigma_{\Gamma_0} \colon Y_{\sigma_{\Gamma_0}} \to \C^2$ coincides with the minimal sequence of point blowups of $(\C^2,0)$ such that $\widetilde{\ell} \big(V (\partial W)\big) \subset V(\Gamma_{\sigma_{\Gamma_0}})$;
\item \label{prop:morphism_graph_local_2}
	the restriction $\widetilde\ell|_{\Gamma_0} \colon \Gamma_0 \to \NL(\C^2,0)$ induces an isomorphism of graphs from $\Gamma_0$ onto its image, which is the subgraph of $\Gamma_{\sigma_{\Gamma_0}}$ whose underlying topological space is the closure of the connected component of  $\Gamma_{\sigma_{\Gamma_0}}\setminus  \widetilde{\ell}(\partial W)$ which contains $ \widetilde{\ell}(  W)$; 
\item \label{prop:morphism_graph_local_3}
	the isomorphism of (ii) respects the weights of all vertices of $\Gamma_0$ that are contained in $W$.
 \end{enumerate}
\end{proposition}

Let $v$ be a vertex of $\Gamma_{\pi}$ which is not a $\cal{P}$-node, let $W$ be the connected component of $\Gamma_\pi\setminus\{\cal P\text{-nodes}\}$ containing $v$, and let $\Gamma_0$ be any subgraph of $\Gamma_\pi$ contained in the closure of $W$ and such that $v\in V(\Gamma_0)$ and all edges of $\Gamma_\pi$ at $v$ are edges of $\Gamma_0$.
Then the last part of the statement of Proposition~\ref{prop:morphism_graph_local} tells us that $g(E_v)=g(E_{\widetilde\ell(v)})=0$ and $E_v^2=E_{\widetilde\ell(v)}^2$, where the self-intersection of $E_v$ is computed in $X_\pi$ and the one of $E_{\widetilde\ell(v)}$ is computed in $Y_{\sigma_{\Gamma_0}}$.

\begin{proof}
Write $\partial W = \{z_1,\ldots,z_n\}\subset V(\Gamma_\pi)$.	
Observe that, since $\Gamma_0$ contains no $\cal P$-node in its interior, $\widetilde\ell$ does not fold it, and thus $\partial\big(\widetilde\ell(W)\big) = \widetilde\ell(\partial W)$ as subsets of $V(\Gamma_{\sigma_{\Gamma_0}})$.
If $W$ contains at least one vertex of $\Gamma_0$, that is if $\Gamma_0$ has at least a vertex which is not a point of $S$, denote by $\Gamma^\circ_0$   the maximal subgraph of $\Gamma_0$ contained in $W$ and set $V(\Gamma^\circ_0) = \{v_1, \ldots, v_r\}$, so that $V(\Gamma_0) = \{v_1, \ldots, v_r,z_1,\ldots,z_n\}$.
Let $U$ be a tubular neighborhood of the curve $C= E_{v_1} \cup \ldots  \cup E_{v_r}$ in $X_{\pi}$. 
Since the incidence matrix of $\Gamma^\circ_0$ is negative definite, the analytic contraction $\eta \colon (U,C) \to (S,p)$ of the curve $C$ onto a point $p$ defines a normal surface singularity $(S,p)$.
Observe that, since $\pi$ is the minimal resolution of $(X,0)$ which factors through its Nash transform, the only components of $\pi^{-1}(0)$ that could be contracted while retaining smoothness of the ambient surface are associated with $\cal P$-nodes of $(X,0)$.
Since $\Gamma^\circ_0\subset W \subset \Gamma_{\pi}\setminus S$ contains no $\cal P$-node, this implies that $\eta \colon (U,C) \to (S,p)$ is the minimal good resolution of the surface germ $(S,p)$. 
On the other hand, if $W$ contains no vertex of $\Gamma_0$, then $\Gamma_0$ consists of two vertices $v$ and $v'$ in $S$ and a single edge corresponding to an intersection point $p=E_v \cap E_{v'}$, in which case we set $(U,C) = (S,p) = (X_\pi,p)$ and $\eta = \mathrm{Id}_{U}$.

Let $\widehat \pi \colon X_{\widehat \pi} \to X$ be the minimal resolution of $(X,0)$ that factors through its Nash transform and such that $\ell\circ\widehat\pi$ factors through $\sigma_{\Gamma_0}$ via a map $\widehat{\ell}\colon X_{\widehat \pi}\to Y_{\Gamma_0}$. 
Then $\widehat \pi$ factors through $\pi$ by minimality of the latter, so that we obtain a commutative diagram as follows: 
\[
\xymatrix@C=3em{  
	X_{\widehat \pi}   \ar@/^1.2pc/[rr]^{\widehat\pi}    \ar[r]^{\beta}     \ar[d]_{\widehat{\ell}}   
	&  X_{\pi}  \ar[r]^{\pi}  & X  \ar[d]^{\ell}
	\\
	Y_{\Gamma_0} \ar[rr]^{\sigma_{\Gamma_0}} &  & \C^2
}
\]
Set $\widehat{U}=\beta^{-1}(U)$ and $\widehat C = \beta^{-1}(C)$, so that $(\widehat U,\widehat C)$ contracts to $(S,p)$ via $\widehat \pi = \pi\circ \beta$.
Set $U' = \widehat{\ell}(\widehat{U})$ and consider the curve $C' = \widehat \ell \big(\widehat C\big) = E_{w_1} \cup \ldots  \cup E_{w_s} \subset \sigma_{\Gamma_0}^{-1}(0)$ in $Y_{\Gamma_0}$.
Observe that, since $\partial\big(\widetilde\ell(W)\big) = \widetilde\ell(\partial W) \subset V(\Gamma_{\sigma_{\Gamma_0}})$, we have $\{w_1, \ldots, w_s\}=\widetilde{\ell}(W) \cap V(\Gamma_{\sigma_{\Gamma_0}})$.
Moreover, since $\widehat U \cap \widehat \pi^{-1}(0) \subset \widehat C \cup E_{z_1} \cup \cdots \cup E_{z_n}$ and no curve among the $E_{z_i}$ is contracted by $\widehat\ell$, we deduce that $U'$ is open in $Y_{\Gamma_0}$.
It follows that $U'$ is a tubular neighborhood of $C'$.
This also shows that the set $\widetilde\ell(\Gamma_0)$ is the closure of the connected component of  $\Gamma_{\sigma_{\Gamma_0}}\setminus  \widetilde{\ell}(\partial W)$ which contains $ \widetilde{\ell}(W)$.
To establish the proposition, it is then sufficient to prove that $\ell$ induces an isomorphism between the pairs $(U,C)$ and $(U',C')$.

Similarly as above, the contraction of the curve $C'$ in $U'$ defines a normal surface singularity $(S',p')$ and an analytic map $\eta' \colon (U',C') \to (S',p')$ which is a good resolution of $(S',p')$.
Moreover, the restriction $\widehat{\ell}|_{U}$ induces a finite analytic map $\widecheck{\ell} \colon (S,p)\to (S',p')$.
Since $W$ contains no $\cal P$-node, by Lemma \ref{lem:projections_LNE} we have $\deg(w) = 1$ for every divisorial point $w$ of $W \cap \Gamma_0$.
This implies that  $\widecheck{\ell}$ is a one-sheeted analytic covering between normal complex analytic spaces, thus an isomorphism by \cite[Proposition 14.7]{Remmert1994}.
It follows that $\widecheck{\ell}^{-1} \circ \eta'$ is a good resolution of $(S,p)$.
Therefore, by minimality of the resolution $\eta$, the map $\widecheck{\ell}^{-1} \circ \eta'$ factors through $\eta$ via a finite sequence of point blowups $\alpha \colon (U',C') \to (U,C)$, so that we obtain the following commutative diagram:
\[
\xymatrix@C=4em{  
	(\widehat{U},\widehat{C})   \ar[r]^(.43){\beta}  \ar[d]_{\widehat{\ell}}   
	&  (U, C)      \ar[r]^{\eta}  & (S,p) 
	\\
	( U',C')   \ar[ru]^{\alpha}   \ar[rr]^{\eta'} &  &(S',q) \ar[u]_{\widecheck{\ell}^{-1}}
}
\]
It remains to show that $\alpha$ is an isomorphism.
If this is not the case, then the exceptional component of the last point blowup forming $\alpha$, which is contractible by definition, is the image through $\widehat\ell$ of the exceptional component of one of the point blowups forming $\beta$.
As this contradicts the minimality condition in the definition of $\sigma_{\Gamma_0}$, this proves both parts~\ref{prop:morphism_graph_local_2} and \ref{prop:morphism_graph_local_3} of the proposition. Part~\ref{prop:morphism_graph_local_1} is then a consequence of the minimality of $\pi$.
\end{proof}

What might prevent Proposition~\ref{prop:morphism_graph_local} from holding globally on $\Gamma_\pi$ is that, for example, there might exist an edge $e$ of $\Gamma_\pi$ such that $\widetilde\ell(e)$ contains in its interior one (and, for the sake of the example, exactly one) point of the form $\widetilde\ell(v)$ for some vertex $v$ elsewhere in $\Gamma_\pi$.
However, whenever this happens it is always possible to \emph{refine} the graph $\Gamma_\pi$, performing a blowup of the double point of the exceptional divisor of $X_\pi$ that corresponds to $e$ and thus subdividing the edge $e$ by adding a new vertex $w$, and this vertex satisfies $\widetilde\ell(w)=\widetilde\ell(v)$.
Observe that, if $(X,0)$ were arbitrary, this could still fail to give a morphism of graphs since the vertex associated with the blowup would not necessarily be sent to $\widetilde\ell(v)$ by $\widetilde\ell$.
The fact that this does not occur in the case of LNE surfaces, and that therefore we can refine $\Gamma_\pi$ to obtain a morphism of graphs, is the content of the following corollary.

\begin{corollary}\label{cor:morphism_graphs}
	Let $(X,0)$ be a LNE normal surface germ, let $\pi\colon X_\pi\to X$ be the minimal good resolution of $(X,0)$ that factors through its Nash transform, let $\ell\colon (X,0)\to (\C^2,0)$ be a generic projection, let $\widetilde\ell\colon \NL(X,0)\to\NL(\C^2,0)$ be the map induced by $\ell$, and let $\sigma_\ell \colon Y_{\sigma_{\ell}} \to \C^2$ be the minimal sequence of point blowups of $(\C^2,0)$ such that $\widetilde\ell\big(V(\Gamma_\pi)\big)\subset V(\Gamma_{\sigma_{\ell}})$.
	Then there exists a good resolution $\pi'\colon X_{\pi'}\to X$ of $(X,0)$, obtained by composing $\pi$ with a finite sequence of blowups of double points of the successive exceptional divisors, such that $\widetilde\ell$ induces a morphism of graphs $\widetilde\ell |_{\Gamma_{\pi'}}\colon \Gamma_{\pi'}   \to \Gamma_{\sigma_{\ell}}$.
\end{corollary}

\begin{proof}
	If $\Gamma_\pi$ consists of a single vertex then there is nothing to prove, so that we can assume without loss of generality that $\Gamma_\pi$ has at least an edge.
	Let $e$ be an edge of $\Gamma_\pi$ and let $\Gamma_0$ be the subgraph of $\Gamma_\pi$ that consists of $e$ and of the two vertices $v$ and $v'$ to which the latter is adjacent.
	Let $\sigma_{\Gamma_0}$ be the minimal sequence of point blowups of $(\C^2,0)$ such that $\widetilde\ell(v),\widetilde\ell(v')\in V(\Gamma_{\sigma_{\Gamma_0}})$.
	By Proposition~\ref{prop:morphism_graph_local}, $\widetilde\ell(e)$ is an edge of $\Gamma_{\sigma_{\Gamma_0}}$, and in particular $\widetilde\ell$ induces an isomorphism of smooth germs $\alpha\colon (X_\pi,E_v\cap E_{v'})\stackrel{\sim}{\longrightarrow} (Y_{\sigma_{\Gamma_0}},E_{\widetilde\ell(v)}\cap E_{\widetilde\ell(v')})$.
	Now, $\sigma_{\ell}$ factors through $\sigma_{\Gamma_0}$ by minimality of the latter.
	In particular, a finite sequence of point blowups above $E_{\widetilde\ell(v)}\cap E_{\widetilde\ell(v')}$ occur in this factorization.
	By performing the same sequence of blowups on $(X_\pi,E_v\cap E_{v'})$ via the isomorphism $\alpha$, we subdivide the edge $e$ in a chain of edges that is sent isomorphically to a subgraph of $\Gamma_{\sigma_\ell}$ via $\widetilde\ell$.	
	Repeating this procedure for every edge $e$ of $\Gamma_\pi$, we obtain the resolution $\pi'$ that we were after.
\end{proof}

Observe that the resulting morphism of graphs $\widetilde\ell|_{\Gamma_{\pi'}} \colon \Gamma_{\pi'} \to \Gamma_{\sigma_\ell}$ is not surjective, as is clear from Example~\ref{ex:pi_not_minimal_multiplicities}.
This issue will be discussed further in Section~\ref{sec:discriminant}.

\smallskip

	We conclude the section by discussing a remarkable property of the Nash transforms of LNE normal surface germs.
	The singularity $(S,p)$ appearing in the course of the proof of Proposition~\ref{prop:morphism_graph_local}, being isomorphic to the singularity $(S',p')$ appearing in a modification of $(\C^2,0)$, is \emph{sandwiched}, which means that it admits a proper bimeromorphic morphism to a smooth surface germ.
	Sandwiched singularities play an important role in Spivakovsky's proof of resolution of singularities of surfaces via normalized Nash transforms \cite{Spivakovsky1990}, since Hironaka \cite{Hironaka1983} proved that it is possible to reduce any singularity to a sandwiched singularity by a finite sequence of normalized Nash transforms.
	Since Proposition~\ref{prop:morphism_graph_local} applies in particular to any connected component of the complement in $\NL(X,0)$ of its $\cal P$-nodes, and the $\cal P$-nodes are precisely the divisorial valuations corresponding to the exceptional components of the Nash transform of $(X,0)$, we deduce the following result:

\begin{corollary}
	Let $(X,0)$ be a LNE normal surface germ and let $\nu\colon \mathcal N(X) \to X$ be the Nash transform of $(X,0)$.
	Then all the singularities of $\mathcal N(X)$ are sandwiched.
\end{corollary}


\section{Inner rates on LNE surface germs}
\label{sec:inner_rates}

In this section we move to the study of the inner rates of LNE surface germs, whose definition was recalled immediately before the proof of Lemma~\ref{lem:projections_LNE}, and prove parts \ref{thm:main_inner_rates} and \ref{thm:main_P-vector} of Theorem~\ref{thm:main}.

\medskip

We begin by endowing the dual graph of a good resolution of $(X,0)$ with a natural metric.
Let $\pi\colon X_\pi\to X$ be a good resolution of $(X,0)$ factoring through the blowup of its maximal ideal, and denote by $|\Gamma_\pi|$ the topological space underlying the graph $\Gamma_\pi$.
We endow $|\Gamma_\pi|$ with the metric defined by declaring that the length of an edge connecting two vertices $v$ and $w$ is equal to $1/\lcm(m_v,m_w)$, and denote by $d$ the associated distance function.
Observe that, since the exceptional component of the blowup of an intersection point between the two components associated with $v$ and $w$ has multiplicity $m_v+m_w$, and $1/\lcm(m_v,m_w)=1/\lcm(m_v,m_v+m_w)+1/\lcm(m_v+m_w,m_w)$, the metric on $|\Gamma_\pi|$ is compatible with subdividing the edges of the graph $\Gamma_\pi$ by blowing up $X_\pi$ at double points of $\pi^{-1}(0)$, and thus induces a metric on $\NL(X,0)$.
The reader should be warned that this metric on $\Gamma_\pi$ is not the same as the one defined in \cite[\S2.1]{BelottodaSilvaFantiniPichon2019}, albeit it is strictly related to the latter and was already briefly used in Lemma~5.5 of \emph{loc.\ cit.}

The following proposition is strictly stronger than part \ref{thm:main_inner_rates} of Theorem~\ref{thm:main}, as it computes inner rates on the whole $\NL(X,0)$ rather than on a specific resolution graph.

\begin{proposition}
	\label{prop:inner_rate_LNE}
	Let $(X,0)$ be a LNE normal surface germ.
	Then, for every divisorial point $v$ of $\NL(X,0)$, the inner rate $q_v$ of $v$ equals $d(v, V_{\cal L})+1$, where $d(v, V_{\cal L})$ denotes the distance of $v$ from the set $V_{\cal L}$ of all $\cal L$-nodes of $(X,0)$.
\end{proposition}

\begin{proof}
Let $\pi \colon X_{\pi} \to X$ be the minimal good resolution of $(X,0)$ which factors through its Nash transform and let $\Gamma_{\pi'}$ be a refinement of $\Gamma_\pi$ as in Corollary~\ref{cor:morphism_graphs}.
We will begin by proving the wanted equality for divisorial points contained in $\Gamma_{\pi'}$.
Denote by $w_0$ the unique $\cal L$-node of $(\C^2,0)$, that is the divisorial point associated with the blowup of $\C^2$ at $0$.
For every divisorial point $w$ of $\NL(\C^2,0)$ the inner rate of $w$ is $d(w,w_0)+1$ by \cite[Lemma~5.5]{BelottodaSilvaFantiniPichon2019} (or, in a more elementary way, by a simple computation using Lemma~3.6 of \emph{loc.\ cit.}).
Since the inner rates on $(X,0)$ and $(\C^2,0)$ commute with the map $\widetilde\ell$ (see \cite[Lemma~3.2]{BelottodaSilvaFantiniPichon2019}), we need to show that $d(v,V_{\cal L}) = d(\widetilde\ell(v),w_0)$.
We claim that, if $\gamma$ is an injective path in $\Gamma_{\pi'}$ connecting two divisorial points $v_1$ and $v_2$, then the length of $\gamma$ is greater or equal to the length of its image $\widetilde\ell(\gamma)$ in $\NL(\C^2,0)$, with equality holding as long as $\widetilde\ell$ maps $\gamma$ injectively onto its image.
Indeed, any edge $e$ in $\gamma$ is sent via $\widetilde\ell$ to an edge $\widetilde\ell(e)$ of the dual graph of some sequence of point blowups of $(\C^2,0)$ thanks to Corollary~\ref{cor:morphism_graphs}.
It then follows from Lemma~\ref{lem:projections_LNE}.\ref{lem:projections_LNE_multiplicities} that the edges $e$ and $\widetilde\ell(e)$ have the same length, which implies our claim.
In particular, since $V_{\cal L}=\widetilde\ell^{-1}(w_0)$, we deduce that $d(v,V_{\cal L})\geq d(\widetilde\ell(v),w_0)$.
To obtain the converse inequality it is sufficient to prove that there exists a path $\gamma$ from $v$ to an element of $V_{\cal L}$ where $\widetilde\ell$ is injective.
This follows from the fact that there exists such a path along which the inner rate function is strictly decreasing (and hence injective), which was proven in \cite[Proposition~3.9]{BelottodaSilvaFantiniPichon2019}.
The fact that the equality holds on the whole of $\NL(X,0)$ is a consequence of \cite[Lemma~5.5]{BelottodaSilvaFantiniPichon2019} (which is itself based on the same computations using Lemma~3.6 of \emph{loc.\ cit.} that appears above).
\end{proof}

\begin{remark}
	Proposition~\ref{prop:inner_rate_LNE} shows that the inner rate function generalizes the function $s$ used by Spivakovsky in \cite[Definition~5.1]{Spivakovsky1990} to study minimal and sandwiched surface singularities.
\end{remark}
	
In order to prove part \ref{thm:main_P-vector} of Theorem~\ref{thm:main}, we need to rely on a deeper result, the so-called \emph{Laplacian formula} for the inner rate function that we obtained in \cite{BelottodaSilvaFantiniPichon2019} and that we will briefly recall now.
In order to state this formula we will introduce two additional vectors indexed by the vertices of the dual graph $\Gamma_\pi$ of a good resolution $\pi\colon X_\pi\to X$ of $(X,0)$.
Let $L_\pi$ and $P_\pi$ be respectively the $\cal L$- and the $\cal P$-vector of $(X,0)$ as before.
For every vertex $v$ of $\Gamma_\pi$, set $k_v=\val_{\Gamma_\pi}(v)+2g(v)-2$ and $a_v=m_vq_v$, and consider the vectors $K_\pi=(k_v)_{v\in V(\Gamma_\pi)}$ and $A_\pi=(a_v)_{v\in V(\Gamma_\pi)}$.
Denote by $I_{\Gamma{\pi}}$ the incidence matrix of the exceptional divisor of $\pi$.
Then the following equality holds:
\begin{equation}\label{equation:laplacian_formula_effective}
I_{\Gamma_\pi} \cdot A_\pi = K_\pi  +  L_\pi - P_\pi.
\end{equation}
This equality is an effective version (see \cite[Proposition~5.3]{BelottodaSilvaFantiniPichon2019}) of the main result of \emph{loc.\ cit.}

\begin{proof}[Proof of part \ref{thm:main_P-vector} of Theorem~\ref{thm:main}]
	For every vertex $v$ of $\Gamma_\pi$, equation~\eqref{equation:laplacian_formula_effective} yields
\begin{equation}\label{eq:laplacian_in_proof}
m_v q_v E_v^2 + \sum\nolimits_{v'} m_{v'}q_{v'} = \val_{\Gamma_\pi}(v) + 2g(E_v) - 2 + l_v - p_v\,,
\end{equation}
where the sum runs over the vertices $v'$ of $\Gamma_\pi$ adjacent to $v$.
Then the equality we are after follows from the fact that
$
E_v\cdot \sum_{v'\in V(\Gamma_\pi)} E_{v'} = E_v^2 + \val_{\Gamma_\pi}(v),
$
that $l_v = - E_v\cdot Z_{\max}(X,0)$ by definition of $l_v$,
that $Z_{\max}(X,0)=Z_{\min}$ by part \ref{thm:main_L-vector} of the theorem, and that
$
E_v\cdot Z_{\Gamma_\pi} = - E_v^2 + 2g(E_v) - 2
$
by definition of $Z_{\Gamma_\pi}$.
Whenever $v$ is an $\cal L$-node, we have $q_v=1$ and $ {q_{v'}}  =1+1/{ {m_{v'}}}$ by Proposition~\ref{prop:inner_rate_LNE}.
Therefore, the left hand side of equation~\eqref{eq:laplacian_in_proof} becomes equal to $E_v\cdot Z_{\max}(X,0) + \sum\nolimits_{v'} 1=-l_v+\val_{\Gamma_\pi}(v)$, and we deduce that $p_v= 2\big(g(E_v) +l_v - 1\big)$.
\end{proof}

\section{End of the proof of Theorem~\ref{thm:main}}
\label{sec:end_proof_main}

In this section we conclude the proof of Theorem~\ref{thm:main}, showing parts \ref{thm:main_P-nodes} and \ref{thm:main_Nash_transform}, which means that we are interested in determining the $\cal P$-nodes of the LNE surface germ $(X,0)$.

\medskip 

We begin with two definitions.
Let $\pi$ denote a good resolution of $(X,0)$ that factors through the blowup of its maximal ideal and through its Nash transform, let $v$ be a vertex of $\Gamma_\pi$, and let $e=[v,v']$ be an edge of $\Gamma_\pi$ adjacent to $v$.
We say that $e$ is \emph{incoming} at $v$ if we have $q_v>q_{v'}$.
Following \cite[Definition~5.3]{Spivakovsky1990}, we say that it is a \emph{central node} of $\Gamma_\pi$ if $v$ has at least two incoming edges.

Observe that the $\cal L$-nodes of $\Gamma_\pi$ have no incoming edges, and that the number of incoming edges at a vertex $v$ does not depend on the choice of a resolution such that $v$ is a vertex of the associated dual graph, since the inner rates increase along any new edge introduced by blowing up a smooth point.
In the LNE case, we can prove the following more precise result, building on the local degree formula \cite[Lemma 4.18]{BelottodaSilvaFantiniPichon2019}.

\begin{lemma} \label{lem:degree_inward_edges}
	Let $(X,0)$ be a LNE normal surface germ, let $\pi$ be a good resolution of $(X,0)$ that factors through its Nash transform, and let $v$ be a vertex of $\Gamma_\pi$.
	Then the local degree $\deg(v)$ at $v$ equals $l_v$ if $v$ is an $\cal L$-node of $\Gamma_\pi$, or the number of incoming edges of $\Gamma_\pi$ at $v$ otherwise.
\end{lemma}

\begin{proof}
Denote by $\ell :(X,0) \to (\C^2,0)$ a generic projection. 
Assume first that $v$ is an $\cal L$-node. 
In this case, we can compute the degree directly via the definition, using a generic linear form $h:(X,0) \to (\C,0)$ that factors through $\ell$ (that is, such that there exists a linear projection $\ell_h:(\C^2,0)\to (\C,0)$ satisfying $h= \ell_h \circ \ell$). 
More precisely, let $\gamma$ be the curve $h^{-1}(0) \cap X$. Since $h$ factors through $\ell$, we know that $\ell(\gamma)$ is a line in $(\C^2,0)$. 
Now, $l_v$ corresponds to the number of distinct irreducible components of the strict transform of $\gamma$ by $\pi$ that intersect $E_v$. 
Since each of those components is smooth, we conclude that $\deg (v)=l_v$ by the definition of degree.
Assume now that $v$ is not an $\cal L$-node. 
By Corollary~\ref{cor:morphism_graphs}, there exist a refinement $\Gamma_{\pi'}$ of $\Gamma_\pi$ and a sequence of point blowups $\sigma:Y \to \C^2$ such that $\widetilde{\ell}$ induces a morphism of graphs $\Gamma_{\pi'} \to \Gamma_{\sigma}$.
In particular, since $\widetilde\ell$ respects inner rates (see again \cite[Lemma~3.2]{BelottodaSilvaFantiniPichon2019}), all edges of $\Gamma_{\pi'}$ that are incoming at $v$ are sent to the unique edge of $\Gamma_\sigma$ that is incoming at $\widetilde\ell(v)$ (its uniqueness can for example be seen as a consequence of Proposition~\ref{prop:inner_rate_LNE}).
The lemma follows now by applying the formula of \cite[Lemma 4.18]{BelottodaSilvaFantiniPichon2019}, observing that the local degree $\deg (e)$ along every edge adjacent to $v$ equals 1 by Lemma~\ref{lem:projections_LNE}.\ref{lem:projections_LNE_degree} and that, even if further point blowups may be needed to pass from $\pi'$ to a resolution adapted to $\ell$, no new edge can be incoming at $v$.
\end{proof}

We can now complete the proof of our main theorem.

\begin{proof}[End of proof of Theorem~\ref{thm:main}]
Let $(X,0)$ be a LNE normal surface germ, denote by $\pi \colon X_{\pi} \to X$ be the minimal good resolution of $(X,0)$, let $\pi' \colon X_{\pi'} \to X$ the minimal one that factors through the Nash transform, and let $v$ be a vertex of $\Gamma_{\pi'}$.
By combining the lemmas~\ref{lem:projections_LNE}.\ref{lem:projections_LNE_degree} and \ref{lem:degree_inward_edges}, we obtain the following:
\begin{equation}\label{eqn:star}
\mbox{$v$ is a $\cal P$-node of $\Gamma_{\pi'}$ if and only if either $l_v>1$, or $v$ is a central node of $\Gamma_{\pi'}$}\tag{$*$}
\end{equation}
which establishes part \ref{thm:main_P-nodes} of the theorem. 
We claim that this also implies that the $\cal P$-nodes of $\Gamma_{\pi'}$ are already on the graph $\Gamma_{\pi} \subset \Gamma_{\pi'}$ (possibly in the interior of some edge).
Indeed, if $v$ is a vertex of $\Gamma_{\pi'} \setminus \Gamma_{\pi}$, since $\pi'$ is obtained from $\pi$ by a sequence of point blowups, we deduce from Proposition \ref{prop:inner_rate_LNE} that there is only one incoming edge at $v$ (observe that all $\cal L$-nodes of $(X,0)$ are contained in $\Gamma_\pi$ thanks to Proposition~\ref{prop:reduced tangent cone}), so that the claim follows from \eqref{eqn:star}.
Therefore we obtain $\pi'$ by successive blowups of double points on the exceptional divisor of $\pi$.
Now, let $e$ be an edge of $\Gamma_\pi$.
If $e=[v,v']$ contains no $\cal P$-node, then it is also an edge of $\Gamma_{\pi'}$, and by applying Proposition~\ref{prop:morphism_graph_local} to its closure we deduce that its image through the map induced by a generic projection $\ell\colon (X,0) \to (\C^2,0)$ is an edge $\widetilde\ell(e)=[\widetilde\ell(v),\widetilde\ell(v')]$ of $\Gamma_{\sigma_\ell}$.
Therefore we have $|q_{\widetilde\ell(v)}-q_{\widetilde\ell(v')}|=d\big(\widetilde\ell(v),\widetilde\ell(v')\big)$, as it can for example be seen by Proposition~\ref{prop:inner_rate_LNE}, and since the inner rate map commutes with $\widetilde\ell$, and $d(v,v')=d\big(\widetilde\ell(v),\widetilde\ell(v')\big)$ by Lemma~\ref{lem:projections_LNE}.\ref{lem:projections_LNE_multiplicities}, we deduce that $|q_v-q_{v'}|=d(v,v')$.
This shows that if $|q_v-q_{v'}|<d(v,v')$ then $e$ must contain a $\cal P$-node.
Conversely, if $e$ contains a $\cal P$-node $w$ then, as it can only contain one $\cal P$-node, $e$ is folded in two by the projection $\widetilde\ell$.
It follows that, with respect to the distance $d$, the inner rate grows linearly with slope 1 from $v$ to $w$, and then decreases linearly with slope 1 from $w$ to $v'$, so that $|q_v-q_{v'}|<d(v,v')$.
We also deduce that $d(v,v') = d(v,w) + d(w,v') = (q_w-q_v) + (q_w - q_{v'})$, and therefore $q_w = \big(d(v,v')+q_v+q_{v'}\big)/2$.
This reasoning can be repeated after blowing up the double point of $\pi^{-1}(0)$ corresponding to $e$, and is therefore sufficient to establish part \ref{thm:main_Nash_transform} of Theorem~\ref{thm:main} and thus conclude its proof.
\end{proof}

%
%


\section{Discriminant curves}
\label{sec:discriminant}

In this section we focus our attention on the discriminant curve of a generic plane projection of a LNE normal surface germ.
We describe those curves completely, proving in Theorem~\ref{thm:discriminant_precise} a more precise version of Theorem~\ref{thm:main2} from the introduction.
In order to do so, we need to pursue in greater depth the study of the properties of the map $\widetilde\ell$ already undertaken in Section~\ref{sec:key_lemma_multiplicities}.

\medskip

Let $(X,0)$ be a LNE normal surface singularity, let $\pi \colon X_{\pi} \to X$ be the minimal good resolution of $(X,0)$ which factors through its Nash transform, let $\ell \colon (X,0) \to (\C^2,0)$ be a generic projection, let $\widetilde\ell\colon \NL(X,0)\to \NL(\C^2,0)$ be the induced morphism, and, as in Section~\ref{sec:key_lemma_multiplicities}, let $\sigma_{\ell} \colon Y_{\sigma_\ell} \to \C^2$ be the minimal sequence of point blowups of $(\C^2,0)$ such that $\widetilde\ell\big(V(\Gamma_\pi)\big)\subset V(\Gamma_{\sigma_\ell})$.
We call \emph{$\Delta$-node} of $\Gamma_{\sigma_{\ell}}$ any vertex $v$ which is the image by $\widetilde{\ell}$ of a $\cal P$-node of $\Gamma_\pi$, and we call \emph{root vertex} of $\Gamma_{\sigma_{\ell}}$ the image by $\widetilde{\ell}$ of the $\cal L$-nodes of $\Gamma_\pi$.
Observe that the root vertex of $\Gamma_{\sigma_\ell}$ is the divisorial point associated with the exceptional divisor of blowup of $\C^2$ at $0$, which can also be seen as the unique $\cal L$-node of $(\C^2,0)$.
Moreover, a vertex $v$ of $\Gamma_{\sigma_\ell}$ is a $\Delta$-node if and only if the associated exceptional component $E_v\subset Y_{\sigma_\ell}$ intersects the strict transform $\Delta^*$ of the discriminant curve $\Delta$ of $\ell$ via ${\sigma_\ell}$.

The following proposition explains that for LNE surfaces the morphism $\sigma_\ell$ coincides with the minimal good embedded resolution of the discriminant curve $\Delta$.

\begin{proposition} \label{prop:discriminant_resolution}
	Let $(X,0)$ be a LNE normal surface germ, let $\ell \colon (X,0)\to (\C^2,0)$ be a generic projection of $(X,0)$, and let $\pi\colon X_\pi\to X$ be the minimal good resolution of $(X,0)$ that factors through its Nash transform.
	Consider the three finite sequences of point blowups of $(\C^2,0)$ defined as follows:
	\begin{itemize}
		\item $\sigma_\Delta \colon Y_{\sigma_\Delta} \to \C^2$ is the minimal good embedded resolution of the discriminant curve $\Delta$ associated with $\ell$\,;
		
		\item $\sigma_\Omega \colon Y_{\sigma} \to \C^2$ is the minimal sequence which resolves the base points of the family of projected generic polar curves $\{\ell(\Pi_{\cal D})\}_{\cal D\in \Omega}$\,;
		
		\item $\sigma_\ell \colon Y_{\sigma_\ell} \to \C^2$ is the minimal sequence such that $V(\Gamma_{\sigma_\ell})$ contains $\widetilde\ell\big(V(\Gamma_\pi)\big)$.
	\end{itemize}
	Then $\sigma_\Delta$, $\sigma_\Omega$, and $\sigma_\ell$ coincide.
\end{proposition}

\begin{proof} 
Let us assume for now that the graph $\Gamma_\pi$ does not consist of a single vertex.
We begin by showing that $\sigma_\ell$ and $\sigma_\Omega$ coincide. 
Denote by $S$ be the set of $\cal P$-nodes of $\Gamma_{\pi}$ and by $W_1, \ldots, W_r$ the connected components of $\Gamma_{\pi} \setminus S$.
For each $i=1,\ldots, r$, following the notation of Proposition~\ref{prop:morphism_graph_local}, let $\Gamma_{0,i}$ be the subgraph of $\Gamma_\pi$ induced on the topological closure of $W_i$ in $\Gamma_{\pi}$, and let $\sigma_{\Gamma_{0,i}} \colon Y_{\sigma_{\Gamma_{0,i}}} \to \C^2$ be the minimal sequence of point blowups $(\C^2,0)$ such that $\widetilde{\ell} \big(V (\Gamma_{0,i})\big) \subset V(\Gamma_{\sigma_{\Gamma_{0,i}}})$.
Since $V(\Gamma_\pi)=\bigcup_i V(\Gamma_{0,i})$, it follows that $\sigma_\ell$ is the minimal sequence of point blowups of $(\C^2,0)$ that factors through all the maps $\sigma_{\Gamma_{0,i}}$.
On the other hand, by Proposition~\ref{prop:morphism_graph_local}.\ref{prop:morphism_graph_local_1}, $\sigma_{\Gamma_{0,i}}$ coincides with the minimal sequence $\sigma_{\Omega, i}$ of point blowups of $(\C^2,0)$ such that $\widetilde{\ell} \big(V (\partial W_i)\big) \subset V(\Gamma_{\sigma_{\Gamma_{0,i}}})$.
Since $\bigcup_i \partial W_i=S$, this implies that $\sigma_\ell$ is the minimal sequence of blowups of $\C^2$ over $0$ such that $V(\Gamma_{\sigma_\ell})$ contains the set $\widetilde\ell(S)$ of the $\Delta$-nodes, which is by definition $\sigma_\Omega$.

Let us now prove that $\sigma_\Omega$ factors through $\sigma_\Delta$ by showing that it is a good embedded resolution of the curve $\Delta$.
Assume by contradiction that this is not the case, so that there exist a $\Delta$-node $w$ and a component $\Delta_0$ of $\Delta$ whose strict transform by $\sigma_\Omega$, while intersecting $E_w$ at a smooth point $p$ of an exceptional component $E_w$, is not a curvette of $E_w$.
This implies that the multiplicity of $\Delta_0$ is strictly greater than $m_w$. 
Let $\Pi_0$ be the component of the polar curve $\Pi$ of $\ell$ such that $\Delta_0 = \ell(\Pi_0)$.
Since $\pi$ is minimal, the strict transform of $\Pi_0$ by $\pi$ is a curvette on an exceptional component $E_v$ such that $\widetilde\ell(v)=w$, so that the multiplicity of $\Pi_0$ equals $m_v$.
Since $\widetilde\ell(v)=w$, by Lemma~\ref{lem:projections_LNE}.\ref{lem:projections_LNE_degree} we have $m_{v} = m_w$, and therefore $\mult(\Delta_0) > \mult(\Pi_0)$.
However, the restriction $\ell|_{\Pi_0} \colon \Pi_0 \to \Delta_0$ is a bilipschitz homeomorphism with respect to the outer metric by \cite[pp. 352-354]{Teissier1982}, so that in particular we have $\mult(\Delta_0) = \mult(\Pi_0)$, yielding a contradiction.
This proves that $\sigma_\Omega$ is a good embedded resolution of $\Delta$.

It is now sufficient to show that $\sigma_{\Delta}$ also resolves the base points of the family $\{\ell(\Pi_{\cal D})\}_{\cal D\in \Omega}$, so that it factors through $\sigma_\Omega$.	
Assume by contradiction that this is not the case, so that there exists a component $\Delta_0$ of $\Delta$ whose strict transform by $\sigma_{\Delta}$ meets the exceptional divisor $\sigma_\Delta^{-1}(0)$ at a (smooth) point $p$ which is a base point of the family $\{\ell(\Pi_{\cal D})\}_{\cal D\in \Omega}$. 
Let $w$ be the vertex of $\Gamma_{\sigma_{\Omega}}$ such that $E_w$  is the irreducible component of $\sigma_\Omega^{-1}(0)$ that contains $p$. 
The base point $p$ is resolved by a sequence of point blowups $\delta$ which creates a bamboo (that is, a chain of two-valent vertices ending with a univalent vertex) $B$ living inside $\Gamma_{\sigma_\Omega}\setminus \Gamma_{\sigma_\Delta}$, stemming from the vertex $w$ and having the corresponding $\Delta$-node $w'$ at its extremity.
Since  $\sigma_{\Omega}$ is a resolution of $\Delta$, we can construct a resolution of $(X,0)$ by performing the Hirzebruch--Jung resolution process starting from the morphisms $\ell$ and $\sigma_{\Omega}$. 
For this, we begin by taking the strict transform of $(X,0)$ by the fiber product of $\ell$ and $\sigma_{\Omega}$ and normalize it. 
Since $\ell$ is a cover which ramifies over the discriminant curve $\Delta$, we obtain a normal surface $Z$ and a finite cover ${\ell'} \colon Z \to Y_{\sigma_{\Omega}}$ which ramifies over the total transform $\sigma_{\Omega}^{-1}(\Delta)$ of the discriminant curve $\Delta$. 
Resolving the singularities of $Z$, we obtain a resolution $\pi' \colon X_{\pi'} \to X$ of $(X,0)$ which we can describe as follows.
Since $\Gamma_{\sigma_{\Omega}}$ contains all $\Delta$-nodes, then $\Gamma_{\pi'}$ contains all $\cal P$-nodes of $(X,0)$, and therefore $\pi'$ factors through $\pi$.
Since by the previous part $\sigma_\Omega$ factors through $\sigma_\Delta$, the total transform $\sigma_{\Omega}^{-1}(\Delta)$ has normal crossings in $Y_{\sigma_\Omega}$.
Moreover, each singularity of $Z$ is branched over a double point of  $\sigma_{\Omega}^{-1}(\Delta)$ and has a resolution whose exceptional divisor is a string of rational curves, and the strict transform of the branching locus consists of two curvettes, one at each extremity of the string.
This implies that the bamboo $B$ lifts via $\widetilde{\ell}$ to a bamboo $B'$ in the resolution graph $\Gamma_{\pi'}$ with a $\cal P$-node at its extremity.
This gives a $\cal P$-node with a unique inward edge in $\Gamma_{\pi'}$, and therefore a unique inward edge in $\Gamma_{\pi}$, contradicting the statement~\eqref{eqn:star} appearing on page~\pageref{eqn:star}.
	
In the special case where $\Gamma_\pi$ consists of a single vertex, then the morphisms $\sigma_\Delta$, $\sigma_\Omega$, and $\sigma_\ell$ all coincide with a single blowup of $\mathbb C^2$ along its origin.
To see this, the only part which is not immediate is the factorization of $\sigma_\Omega$ through $\sigma_\Delta$, but the argument given above remains valid in this special case, and thus the proof of the proposition is complete.
\end{proof}

Recall that, as showed by Example~\ref{ex:pi_not_minimal_multiplicities}, the map $\widetilde{\ell} \colon  \Gamma_{\pi} \to \Gamma_{\sigma_{\ell}}$ may fail to be surjective.
As a first step to better understand the situation, the following proposition, which refines the techniques we employed in the course of the proof of Lemma \ref{lem:projections_LNE}, allows us to describe $\widetilde\ell$ more explicitly.

\begin{lemma}\label{lem:projections_LNE_paths}
Let $(X,0)$ be a LNE normal surface germ, let $\ell\colon (X,0)\to (\C^2,0)$ be a generic projection, let $\pi \colon X_{\pi} \to X$ be the minimal good resolution of $(X,0)$ which factors through its Nash transform, and let $v$ and $v'$ be two divisorial points of $\Gamma_{\pi} \subset \NL(X,0)$.
Then $\widetilde\ell(v)=\widetilde\ell(v')$ if and only if the two following conditions are satisfied:
\begin{enumerate}
\item \label{lem:projections_LNE_paths_condition1} $q_v = q_{v'}$\,;
	
\item \label{lem:projections_LNE_paths_condition2} there exists a path $\tau$ in $\Gamma_{\pi}$ between $v$ and $v'$ such that the inner rate of any point in $\tau$ is greater than or equal to $q_{v}$.
\end{enumerate}
\end{lemma}
\begin{proof}
Let us begin by proving the ``only if'' part of the statement.
Assume that $\widetilde{\ell}(v) = \widetilde{\ell}(v')$. 
Then $q_v=q_{v'}$ since both inner rates are equal to $q_{\widetilde\ell(v)}$ by \cite[Lemma~3.2]{BelottodaSilvaFantiniPichon2019}. 
Assume by contradiction that the condition \ref{lem:projections_LNE_paths_condition2} is not satisfied, and let $\gamma$ be a curve in $(\C^2,0)$ which is the image through $\sigma_{\ell}$ of a curvette of $E_{\widetilde{\ell}(v)}$. 
Let $\widehat{\gamma}$ (respectively $\widehat{\gamma}'$) be a component of $\ell^{-1}(\gamma)$ which is the image of a curvette of $E_v$ (respectively $E_{v'}$) via a suitable resolution factoring through $\pi$.
Since $\Gamma_\pi$ is path connected but \ref{lem:projections_LNE_paths_condition2} is not satisfied, then the inner contact $q_{\inn}(\widehat{\gamma}, \widehat{\gamma}')$ between $\widehat{\gamma}$ and $\widehat{\gamma}'$ is strictly smaller than $q_v$ by \cite[Proposition 15.3]{NeumannPedersenPichon2020a}.
On the other hand, by taking a different generic projection $\ell_{\cal{D}'}$ which is also generic with respect to $\widehat{\gamma}\cup\widehat{\gamma}'$, and observing (in a similar way as in the proof of Lemma~\ref{lem:projections_LNE}) that $\ell_{\cal D'}$ induces a bilipschitz homeomorphism for the outer metric from $\widehat{\gamma}\cup \widehat{\gamma}'$ onto its image by \cite[pp. 352-354]{Teissier1982}, we deduce that the outer contact $q_\mathrm{out}(\widehat{\gamma},\widehat{\gamma}')$  between $\widehat{\gamma}$ and $\widehat{\gamma}'$ equals $q_\mathrm{out}\big(\ell_{\cal{D}'}(\widehat{\gamma}),\ell_{\cal{D}'}(\widehat{\gamma}')\big)=q_\mathrm{inn}\big(\ell_{\cal{D}'}(\widehat{\gamma}),\ell_{\cal{D}'}(\widehat{\gamma}')\big)$.
By Lemma~\ref{lem:generic projection} we have $\widetilde\ell_{\cal D'}(v)=\widetilde\ell_{\cal D'}(v')$, and thus the curves $\ell_{\cal{D}'}(\widehat{\gamma})$ and $\ell_{\cal{D}'}(\widehat{\gamma}')$ lift to the same divisor $E_{\widetilde\ell_{\cal D'}(v)}$. 
Therefore their inner contact is bigger than or equal to $q_{\widetilde\ell_{\cal D'}(v)}=q_v$, which contradicts the fact that $(X,0)$ is LNE.

To prove the converse implication, observe that if $w$ and $w'$ are two points of $\NL(\C^2,0)$ then there exists a unique injective path $\tau_{w,w'}$ between $w$ and $w'$ in $\NL(\C^2,0)$, since the latter is a connected infinite tree.
Moreover, for each point $w''$ in the interior of $\tau_{w,w'}$ we have $q_{w''}<\max \{q_w,q_{w'}\}$ (for example, this can be derived from Proposition~\ref{prop:inner_rate_LNE}).
Now assume that $v$ and $v'$ are two points of $\Gamma_\pi$ that satisfy the conditions \ref{lem:projections_LNE_paths_condition1} and \ref{lem:projections_LNE_paths_condition2} and let $\tau$ be any path in $\Gamma_\pi$ between $v$ and $v'$.
By continuity of the projection, $\widetilde\ell(\tau)$ must contain $\tau_{\widetilde{\ell}(v), \widetilde{\ell}(v')}$, and the latter has nonempty interior as soon as $\widetilde{\ell}(v) \neq \widetilde{\ell}(v')$, therefore we deduce that when this is the case then $\tau$ contains a point $v''$ mapping to the interior of $\tau_{\widetilde{\ell}(v), \widetilde{\ell}(v')}$, so that $ {q_{v''}}=q_{\widetilde\ell(v'')}<q_{\widetilde\ell(v)}=q_v$.
As this would contradict condition \ref{lem:projections_LNE_paths_condition2}, we must have $\widetilde{\ell}(v) = \widetilde{\ell}(v')$. 
\end{proof}

To make good use of this result we need to introduce some additional notation.
We denote by $V_N(\Gamma_\pi)$ the set of \emph{nodes} of $\Gamma_\pi$, that is the subset of $V(\Gamma_\pi)$ consisting of the $\cal P$-nodes, the $\cal L$-nodes, and of all the vertices that have valency at least three in $\Gamma_\pi$ (that is, those with at least three adjacent edges).
%
%
Similarly, we call {\it node} of $\Gamma_{\sigma_{\ell}}$ a vertex which is either the root vertex, a $\Delta$-node, or a vertex of valency at least three in $\Gamma_{\sigma_\ell}$, and we denote by $V_N(\Gamma_{\sigma_{\ell}})$ the set of nodes of $\Gamma_{\sigma_{\ell}}$. 

Let $\Gamma$ be either of the two graphs $\Gamma_\pi$ or $\Gamma_{\sigma_\ell}$.
We call \emph{principal part} of $\Gamma$ the subgraph $\Gamma'$ of $\Gamma$ generated by the set $V_N(\Gamma)$ of nodes of $\Gamma$, that is the subgraph defined as the union of all injective paths connecting pairs of points of $V_N(\Gamma)$.
The closure of each component of $\Gamma \setminus \Gamma'$ is a \emph{bamboo} (that is, a chain of valency 2 vertices ending with a valency 1 vertex) stemming from a node of $\Gamma$.

Lemma~\ref{lem:projections_LNE_paths} prompts us to consider an equivalence relation $\sim$ on the graph $\Gamma_{\pi}'$ defined by declaring that two vertices $v$ and $v'$ of $\Gamma_{\pi}'$ are equivalent if the two conditions \ref{lem:projections_LNE_paths_condition1} and \ref{lem:projections_LNE_paths_condition2} of the lemma hold, and two edges $e=[v_1,v_2]$ and $e'=[v_1',v_2']$ are equivalent if and only if $v_1\sim v_1' $ and $v_2\sim v_2'$.
The following proposition relates the nodes of $\Gamma_\pi$ to the ones of $\Gamma_{\sigma_\ell}$ and explains how the equivalence relation $\sim$ allows to retrieve the principal part $\Gamma'_{\sigma_\ell}$ of $\Gamma_{\sigma_\ell}$ from the one of $\Gamma_{\pi}$. 

\begin{proposition}  \label{prop:nodes}  
	Let $(X,0)$, $\pi$, $\ell$, and $\sigma_\ell$ be as above.
	Then we have:
	\begin{enumerate}
		\item	\label{prop:nodes_vertices}
			\(
			\widetilde{\ell}\big(V_N(\Gamma_{\pi})\big)  = V_N(\Gamma_{\sigma_{\ell}})\,;
			\)
		\item	\label{prop:nodes_quotient}
			the map $\widetilde\ell|_{\Gamma'_{\pi}}  \colon \Gamma'_{\pi} \to  \Gamma'_{\sigma_{\ell}}$ identifies the graph $\Gamma'_{\sigma_{\ell}}$ with the quotient graph $\Gamma'_{\pi}/\hspace{-.25em}\sim$. 
	\end{enumerate}
\end{proposition}

\begin{proof}
Observe that $\Gamma_{\sigma_\ell}$ and $\Gamma_{\sigma_\Delta}$ coincide thanks to Proposition~\ref{prop:discriminant_resolution}, and the latter, being the minimal embedded resolution graph of the plane curve $\Delta$, has a very simple shape (see for example \cite{BrieskornKnoerrer1986}).
In particular, we deduce that the principal part $\Gamma_{\sigma_\ell}'$ of $\Gamma_{\sigma_\ell}$ coincides with the union of the injective paths connecting the root vertex of $\Gamma_{\sigma_\ell}$ to one of its $\Delta$-nodes. 
It follows that the image $\widetilde\ell(\Gamma_\pi)$ of $\Gamma_\pi$ via $\widetilde\ell$ contains $\Gamma_{\sigma_\ell}'$.
Indeed, any injective path from an $\cal L$-node to a $\cal P$-node $v$ in $\Gamma_\pi$ is sent by $\widetilde\ell$ to the unique injective path in $\Gamma_{\sigma_\ell}$ from the root vertex to the $\Delta$-node $\widetilde\ell(v)$.
Moreover, any vertex of $\Gamma_{\sigma_\ell}$ contained in $\widetilde\ell(\Gamma_\pi)$ is the image of a vertex of $\Gamma_\pi$, as follows readily from Proposition~\ref{prop:morphism_graph_local}.
From the particular shape of $\Gamma_{\sigma_\ell}$ we also deduce that, if $w$ is a vertex of $\Gamma_{\sigma_\ell}$, then at most one of the edges of $\Gamma_{\sigma_\ell}$ that are outgoing (that is, not incoming in the sense of Section~\ref{sec:end_proof_main}) at $w$ may fail to be contained in the principal part $\Gamma_{\sigma_\ell}'$.

By definition, a vertex $w$ of $\Gamma_{\sigma_\ell}$ is the root vertex (respectively a $\Delta$-node) of $\Gamma_{\sigma_\ell}$ if and only if $\widetilde{\ell}^{-1}(w)$ contains an $\cal L$-node (respectively a $\cal P$-node) of $(X,0)$.
To prove part \ref{prop:nodes_vertices} of the proposition it is then sufficient to establish the following claim: a vertex $w$ of $\Gamma_{\sigma_\ell}$ that is not the root vertex nor a $\Delta$-node has valency at least three if and only if $\widetilde\ell^{-1}(w)$ contains at least a vertex of valency at least three in $\Gamma_\pi$.

The ``if'' part of the claim can be easily obtained by taking a vertex $v$ in $\widetilde\ell^{-1}(w)$ having valency at least three and applying Proposition~\ref{prop:morphism_graph_local} to the subgraph of $\Gamma_\pi$ consisting of the topological closure of $v$ and its adjacent edges.

Let us prove the ``only if'' part of the claim.  
We begin by showing that all edges of $\Gamma_{\sigma_\ell}$ that are outgoing at $w$ belong to $\widetilde\ell(\Gamma_\pi)$.
Assume by contradiction that this is not the case, and let $e$ be such an edge that is not contained in $\widetilde\ell(\Gamma_\pi)$.
In particular $e$ is not contained in $\Gamma_{\sigma_\ell}'$ either, and recalling that $\sigma_\ell$ coincides with the minimal resolution of the plane curve $\Delta$, we deduce that the connected component of $\sigma_\ell\setminus\{w\}$ containing $e$ is a bamboo which cannot be contracted.
Therefore, in the notation of Proposition~\ref{prop:morphism_graph_local}, if $\Gamma_0$ is the graph induced on the closure of any connected component of $\Gamma_\pi\setminus\{\cal P\text{-nodes}\}$ which intersects $\widetilde\ell^{-1}(w)$ nontrivially, then $e$ is contained in the dual graph $\Gamma_{\sigma_{\Gamma_0}}$.
Since $w$ is not a $\Delta$-node of $\Gamma_{\sigma_\ell}$, we deduce from Proposition~\ref{prop:morphism_graph_local} that $e$ is contained in $\widetilde\ell(\Gamma_\pi)$.
To conclude the proof of the claim we will make use of Lemma~\ref{lem:projections_LNE_paths}.
Let $e$ and $e'$ be two distinct outgoing edges of $\Gamma_{\sigma_\ell}$ at $w$.
If there exists a vertex $v$ of $\Gamma_\pi$ such that $e$ and $e'$ are both images of edges of $\Gamma_\pi$ adjacent to $v$, then $v$ has valency at least three in $\Gamma_\pi$, since not being an $\cal L$-node it must also have an incoming edge, and therefore there is nothing else to prove.
We can thus assume without loss of generality that there exist two distinct vertices $v$ and $v'$ of $\Gamma_\pi$ such that $\widetilde\ell(v)=\widetilde\ell(v')=w$ and two edges, $\widetilde e$ adjacent to $v$ and $\widetilde e'$ adjacent to $v'$, whose images contain $e$ and $e'$ respectively.
Since $\widetilde\ell(v)=\widetilde\ell(v')$, it follows from Lemma~\ref{lem:projections_LNE_paths} that there exists an injective path $\tau$ from $v$ to $v'$ passing through outgoing edges only.
However, $\tau$ cannot leave $v$ from $\widetilde e$ and reach $v'$ through $\widetilde e'$, as if that were the case then $\tau$ would become a loop in $\Gamma_{\sigma_\ell}$, which is a tree.
This imply that at least one among the vertices $v$ and $v'$ must have a second outgoing edge, and must therefore be trivalent.
This concludes the proof of the claim, and therefore of part~\ref{prop:nodes_vertices} of the proposition.

To conclude the proof, observe that \ref{prop:nodes_vertices} implies that $\widetilde{\ell}$ restricts to a surjective map $\widetilde\ell|_{\Gamma'_{\pi}} \colon \Gamma'_{\pi} \to \Gamma'_{\sigma_{\ell}}$ between the principal parts of $\Gamma_\pi$ and $\Gamma_{\sigma_\ell}$. 
Lemma~\ref{lem:projections_LNE_paths} then shows that setwise the map $\widetilde\ell$ identifies $\Gamma'_{\sigma_{\ell}}$ with the quotient $\Gamma'_{\pi}/\hspace{-.25em}\sim$.
The fact that it gives an isomorphism of graphs follows then from the fact that, as we already observed, any vertex of $\Gamma_{\sigma_\ell}$ contained in $\widetilde\ell(\Gamma_\pi)$ is the image of a vertex of $\Gamma_\pi$, and so $\widetilde\ell\big(V(\Gamma_\pi')\big) = V(\Gamma_{\sigma_\ell}')$.
\end{proof}

\begin{remarks} 
\begin{enumerate}
	\item Observe that Lemma~\ref{lem:projections_LNE_paths} fails outside of the dual graph of the minimal resolution $\pi$.
	Indeed, it is sufficient to compose $\pi$ with two point blowups, choosing as center two distinct points of the same component of the exceptional divisor of $\pi$ which are identified by a suitable lifting of $\ell$, to obtain two divisorial points $v$ and $v'$ such that $\widetilde\ell(v)=\widetilde\ell(v')$ and for which condition~\ref{lem:projections_LNE_paths_condition2} does not hold.
	
	\item 
	If $w$ is the root vertex of $\Gamma_{\sigma_{\ell}}$, then $\widetilde{\ell}^{-1}(w)$ is exactly the set of $\cal L$-nodes of $(X,0)$, so that $\widetilde{\ell}^{-1}(w) \subset V_N(\Gamma_\pi)$. 
	However, if $w$ is a $\Delta$-node of $\Gamma_{\sigma_{\ell}}$, not all vertices in $\widetilde{\ell}^{-1}(w)$ need to be $\cal P$-nodes of $(X,0)$ (nor, more generally, nodes of $\Gamma_\pi$), as \cite[Example~3.13]{NeumannPedersenPichon2020b} shows.
	If $w$ is a node of $\Gamma_{\sigma_{\ell}}$ which is not a $\Delta$-node and which has valency at least three in $\Gamma_{\sigma_{\ell}}$, we do not know whether $\widetilde{\ell}^{-1}(w)$ may contain vertices having valency less than three in $\Gamma_{\pi}$.
	
	\item In the course of the proof of Proposition~\ref{prop:nodes}, we have shown that $\Gamma_{\sigma_\ell}\setminus \widetilde\ell(\Gamma_\pi)$ consists of bamboos stemming from $\Delta$-nodes.
	Moreover, over a bamboo of $\Gamma_{\sigma_\ell}$ stemming from a vertex $w$ which is not a $\Delta$-node, there is a copy of same bamboo stemming from any vertex $v$ of $\Gamma_\pi$ such that $\widetilde\ell(v)=w$.
	It then follows from Lemma~\ref{lem:projections_LNE_paths} that there can only be one vertex $v$ of $\Gamma_\pi$ which is sent to such a vertex $w$ by $\widetilde\ell$.
\end{enumerate}	
\end{remarks}

We have now collected all the results we need to move to the study of the embedded topological type of the discriminant curve $(\Delta,0)\subset(\C^2,0)$.
Fix once and for all a set of coordinates $(x_1,x_2)$ on $(\C^2,0)$ such that $x_1=0$ is transverse to $\Delta$.
The topological type we are interested in is then completely determined by the characteristic exponents of the Newton--Puiseux expansion with respect to $x_1$ of each branch of $\Delta$ and by the coincident exponents between each pair of branches.
This data is encoded by another combinatorial object, the so-called Eggers--Wall tree $\Theta(\Delta)=\Theta_{x_1}(\Delta)$ of $\Delta$.
We refer the reader to \cite[\S3]{BarrosoPerezPopescu2019} for a thorough introduction to this object, and in particular to Definition~3.8 and Remark~3.14 of \emph{loc.\ cit.} for a formal definition starting from Newton--Puiseux expansions and an interesting historical remark.
From our point of view, it is more convenient to describe the Eggers--Wall tree $\Theta(C)$ of a plane curve germ $(C,0)\subset(\C^2,0)$ starting from the dual graph of a good embedded resolution of $\Delta$ and from the invariants we already consider, namely multiplicities and inner rates.
This follows the philosophy of Section~8 of \emph{loc.\ cit.}, where an embedding of $\Theta(C)$ in a valuation space homeomorphic to $\NL(\C^2,0)$ is described (see in particular Theorem~8.19 there).
The procedure is the following.

\begin{algo}\label{algorithm:EW_tree} 
Denote by $\sigma_C\colon Y_{\sigma_C}\to C^2$ the minimal good embedded resolution of the curve $C$.
The set of nodes $V_N(\Gamma_{\sigma_C})$ of the dual graph $\Gamma_{\sigma_C}$ of $\sigma_C$ is by definition the set consisting of its root, its $C$-nodes, which are the vertices corresponding to the components of ${\sigma_C}^{-1}(0)$ intersecting the strict transform of $C$, and its vertices of valency at least three.
The Eggers--Wall tree $\Theta(C)$ is obtained from the set of nodes $V_N(\Gamma_{\sigma_C})$ of the tree $\Gamma_{\sigma_C}$, from its principal part $\Gamma_{\sigma_C}'$, and from the multiplicities and the inner rates of the vertices of $\Gamma_{\sigma_C}'$, as follows:
\begin{itemize}
	\item From $\Gamma_{\sigma_C}'$, attach one extra edge to the root and one to each $C$-node $w$ for every branch of $C$ passing through $E_w$.
	\item Decorate each node $w\in V_N(\Gamma_{\sigma_C})$ (this includes vertices that have valency larger than three in $\Gamma_{\sigma_C}$ but less than three in $\Gamma_{\sigma_C}'$) with the rational number $e_C(w)=q_w$.
	\item If $e=[w,w']$ is an edge of $\Gamma_{\sigma_C}'$, decorate it with the integer $i(e) = \lcm(m_w,m_{w'})$.
	\item If $e$ is one of the new edges of $\Theta(C)$ adjacent to a vertex $w$, decorate it with the integer $i(e)=m_w$.
\end{itemize}
\end{algo}

The rational numbers $e_C(w)$ on the nodes $w$ on the path connecting the root to a $C$-node $w'$ are then precisely the characteristic exponents of any branch of $C$ passing through $E_w$, while the coincident exponent between two branches can be computed from the functions $e_C$ and $i_C$, as explained in \cite[Theorem~3.25]{BarrosoPerezPopescu2019}.

In order to describe the embedded topological type of the discriminant curve $\Delta$, it remains to show how to combine results we proved in Sections~\ref{sec:key_lemma_multiplicities}, \ref{sec:end_proof_main}, and \ref{sec:discriminant} to determine the input of Algorithm~\ref{algorithm:EW_tree} from the minimal good resolution of $(X,0)$.
This can be done as follows.

\begin{algo}\label{algorithm:descent}
	Denote by $\pi_0\colon X_{\pi_0} \to X$ the minimal good resolution of $(X,0)$.
	Then:
	\begin{itemize}	
	\item The multiplicities and the inner rates of the vertices of $\Gamma_{\pi_0}$ are uniquely determined by parts~\ref{thm:main_L-vector} and \ref{thm:main_inner_rates} of Theorem~\ref{thm:main}.
	
	\item The minimal resolution $\pi\colon X_{\pi} \to X$ of $(X,0)$ factoring through its Nash transform, decorated with its multiplicities and inner rates, is obtained from $\pi_0$ applying the algorithm of part~\ref{thm:main_Nash_transform} of Theorem~\ref{thm:main}.
	This also determines the set of nodes $V_N(\Gamma_\pi)$ of $\Gamma_\pi$ and its principal part $\Gamma_\pi'$.
	
	\item Recall that we have $\sigma_\ell=\sigma_\Delta$ by Proposition~\ref{prop:discriminant_resolution}.
	Therefore, by Propositions~\ref{prop:nodes} we obtain the principal part $\Gamma'_{\sigma_\Delta}$ of $\Gamma_{\sigma_\Delta}$ and the subset $V_N(\Gamma_{\sigma_C})$ consisting of the nodes of $\Gamma_{\sigma_\Delta}$.
	
	\item The multiplicities of the vertices of $\Gamma_{\sigma_\Delta}'$ are determined by the ones of the vertices of $\Gamma_{\pi}$ thanks to Lemma~\ref{lem:projections_LNE}.\ref{lem:projections_LNE_multiplicities}.
	
	\item The inner rates of the vertices of $\Gamma_{\sigma_\Delta}'$ are determined by the ones of the vertices of $\Gamma_{\pi}$ because inner rates commute with $\widetilde\ell$ thanks to \cite[Lemma~3.2]{BelottodaSilvaFantiniPichon2019}.
	\end{itemize}
\end{algo}

We have proven the following result, which is a more precise version of Theorem~\ref{thm:main2} from the introduction.

\begin{theorem} \label{thm:discriminant_precise}
Let $(X,0)$ be a LNE normal surface germ and let $\ell\colon (X,0)\to(\C^2,0)$ be a generic projection.
Then the embedded topology of the discriminant curve $\Delta$ of $\ell$ is completely determined by the topology of $(X,0)$.
More precisely, the Eggers--Wall tree of $\Delta$ can obtained by applying Algorithm~\ref{algorithm:descent} followed by Algorithm~\ref{algorithm:EW_tree} to the dual graph of the minimal resolution of $(X,0)$.
\end{theorem}

\appendix

{ \section{Generic polar curves and Nash transform} \label{appendix:A}

In this appendix we give a comprehensive proof of a result stated in the introduction.

\begin{proposition}\label{proposition:appendix_smooth_transverse}  
	Let $(X,0)$ be a normal surface singularity, let  $\pi\colon X_\pi \to X$  be a good resolution of $(X,0)$,  and let  $h\colon (X,0) \to (\C,0)$ and $\ell \colon (X,0) \to (\C^2,0)$ be respectively a generic linear form and a generic plane projection of $(X,0)$.
	Then:
	\begin{enumerate}
		\item If $\pi$ factors  through the blowup of the maximal ideal of $(X,0)$, then the strict transform via $\pi$ of the hyperplane section $h^{-1}(0)$ associated with $h$ consists of smooth curves intersecting the exceptional divisor $\pi^{-1}(0)$ transversely at smooth points;
		\item If $\pi$ factors  through the Nash transform of $(X,0)$, then the strict transform via $\pi$ of the polar curve of $\ell$ consists of smooth curves intersecting $\pi^{-1}(0)$ transversely at smooth points. 
	\end{enumerate}
\end{proposition}

Observe that in the statement by the word \emph{generic} we also mean generic with respect to $\pi$, which means that the strict transforms by $\pi$ of the associated linear form and polar curve intersect $\pi^{-1}(0)$ only at $\cal L$- and $\cal P$-nodes respectively, rather than moving to other components created by additional blowups (this is a standard condition, see for example \cite[2.2]{BelottodaSilvaFantiniPichon2019}). 
We note that this result seems to be accepted by the experts working in the field (see for instance the discussion of  \cite[Section 2]{Bondil2005} regarding the first part of the proposition), but we have not been able to locate a proof of it in the literature and therefore we provide one here.

\begin{proof} 
  Let us prove a more general version of $(i)$. 
Denote by $\mathfrak M$ the maximal ideal of $(X,0)$ and let $I$ be a $\mathfrak M$-primary ideal (that is, $I$ contains a power of $\mathfrak M$). 
Choose a system of generators $(f_1,\ldots, f_k)$ of $I$, consider the blowup 
$\mathrm{Bl}_I(X) \colon X_I \to X $ of $I$, which is defined as the closure in $X \times \Pro^{k-1}$ of the set $\big\{(x,[f_1(x):\ldots:f_k(x)]) \   \big|\ x \in X\setminus V(I)\big\}$, and let $n_I \colon (\overline{X_I}, \overline{E}_I)\to (X_I,E_I)$ be its normalization. 
Denote by $p_2 \colon  E_I  \to   \Pro^{k-1}$ the projection on the second factor.
By Bertini's theorem, a generic hyperplane $H \colon a_1 z_1 + \ldots + a_k z_k=0$ of $\Pro^{k-1}$ intersects the complex curve $p_2(E_I)$ transversely at smooth points. 
Therefore, the multigerm $(p_2 \circ n_I)^{-1}(H)$ consists of disjoint smooth curves intersecting transversely $\overline{E}_I$ at a finite numbers of smooth points. 
On the other hand, $( p_2 \circ n_I)^{-1}(H)$ is exactly the strict transform of $V(h)$ by $ {
Bl_I(X)} \circ n_I$, where $h$ is the element of $I$ defined by  $h = a_1 f_1 + \ldots + a_k f_k$.
Applying this to $I=\mathfrak M$, since $\pi$ factors through the normalized blowup of the maximal ideal, we obtain part~$(i)$ of the proposition. 

Let us now prove $(ii)$.  
 {Recall that the Nash transform $\nu \colon {\cal N}(X) \to X$ of $(X,0)$ is the projection on the first factor of the closure of $\big\{ (x,T_xX) \,\big|\, x \in X
\setminus\{0\}\big\}$ in $X \times  \Gr(2,\C^k)$. Set $\overline{\nu} = \nu \circ n$  where  $n \colon \overline{\cal N(X)} \to \cal N(X)$ is  the normalization of $\cal N(X)$. 
The Gauss map $X \setminus \{0\} \to \Gr(2,\C^k)$ which sends $x$ to $T_xX$ lifts to a well defined map $ \overline{\lambda} \colon   \overline{\cal N(X)}  \to \Gr(2,\C^k)$. 
To complete the proof of the proposition it is now sufficient to show that the strict transform by $\overline{\nu}$ of a generic polar curve of $(X,0)$ intersects transversely the exceptional divisor $E=\overline{\nu}^{-1}(0)$ at smooth points.}
The proof is similar in spirit to the proof of $(i)$, but relies on a construction from \cite[page 210]{BirbrairNeumannPichon2014}. 
Let $\Omega$ be a dense open subset of $\Gr(k-2, \C^k)$ parametrizing generic plane projections of $(X,0)$, let $\cal D$ be an element of $\Omega$, and let $\ell_{\cal{D}}\colon (X,0)\to (\C^2,0)$ be the generic plane projection of $(X,0)$ associated with $\cal D$. 
Assume first that $k=3$ and denote by $L_{\cal D}$ the projective line in $\Gr(2,\C^3) \cong \Pro^2$ consisting of the $2$-planes of $\C^3$ which contain the line $\cal D$.   
By  Bertini's theorem, $L_{\cal D}$ intersects the complex curve $\overline\lambda(E)$ transversely at smooth points.
Therefore, the multigerm $(\overline{\lambda})^{-1}(L_{\cal D}) = $ consists of disjoint smooth curves intersecting transversely $E$ at a finite numbers of smooth points. 
Since $\Pi_{\cal D} \setminus\{0\}$ is the set of critical points of the restriction $\ell_{\cal D} \colon (X,0) \to (\C^2,0)$ of the linear projection $\C^k \to \C^2$ with kernel $\cal D$ to $(X,0)$, then the strict transform  $\Pi_{\cal D}^*$ by $\overline{\nu}$ is the multigerm defined by $\Pi_{\cal D}^*=\overline\lambda^{-1}(L_{\cal D})$. 
This proves $(ii)$ when $k=3$.  
	
Assume now that $k\ge 3$ and let us choose an $(k-3)$-dimensional subspace $W \subset \C^k$ transverse to    the $2$-planes which belong to the finite set  $B=\overline\lambda(A)$ where $A=\Pi_{\cal D}^* \cap E$.   
Let $\Gr(2,\C^k;W)$ denote the set of $2$-planes in $\C^k$ transverse to $W$, so that the projection $p\colon \C^k\to \C^k/W$ induces a map $p'\colon \Gr(2,\C^k;W) \to \Gr(2,\C^k/W)\cong P^2\C$.
Observe that $\Gr(2,\C^k;W) $ is a Zariski open subset of $\Gr(2,\C^k)$ containing $B$, therefore on a small neighborhood of $A$ in $\overline{\cal N(X)}$ we can define the map $\lambda' = p' \circ \overline{\lambda}$.
We now follow the lines of the argument in the case $k=3$ but using $\lambda'$ instead of $\overline{\lambda}$.
Let $L_{\cal D}$ be the projective line in $ \Gr(2,\C^k/W)\cong P^2\C$ image by $p'$ of the set of $2$-planes that intersect $\cal D$ non-trivially. 
By genericity of $\cal D$ and by Bertini's theorem, $L_{\cal D}$ intersects the complex curve $\overline\lambda(E)$ transversely at smooth points. 
Therefore, the multigerm $\Pi_{\cal D}^* = (\lambda')^{-1}(L_{\cal D})$ consists of disjoint smooth curves intersecting tranversely $E$ at smooth points. 
\end{proof}

\section{A new example of a LNE normal surface singularity}
\label{appendix:example}

The aim of this appendix is to prove the following result:

\begin{proposition}  \label{prop:new LNE} 
	The hypersurface singularity $(X,0)\subset(\C^3,0)$ defined by the equation $x^5+y^5+z^5+xyz=0$ is LNE. 
\end{proposition}

This is a cusp singularity (see \cite{Laufer1977}).
The blowup of $(X,0)$ at $0$ has as exceptional divisor a loop formed by three rational curves, and three singular points (all being $A_1$ singularities) where two of those intersects.
The minimal resolution $\pi \colon X_{\pi} \to X$ is then obtained by composing this blowup with the blowups of those three singular points.
The dual graph $\Gamma_\pi$ of $\pi$ is showed in Figure~\ref{newex_figure 1}, where the negative numbers denote the self-intersections of the corresponding components. 
The vertices $v_1,v_2$, and $v_3$ are the $\cal L$-nodes of $\Gamma_\pi$ and carry one arrow each, representing a component of the strict transform to $X_\pi$ of a generic linear form  $h \colon (X,0) \to (\C^2,0)$.

\begin{figure}[h] 
	\centering
	\begin{tikzpicture}
	\draw[thin,>-stealth,->](0,0)--+(0,0.5);
	\draw[thin,>-stealth,->](-1.6,-2)--+(-0.6,-0.4);
	\draw[thin,>-stealth,->](1.6,-2)--+(0.6,-0.4);
	
	\draw[fill ] (0,0)circle(2pt);
	\draw[fill ] (1.6,-2)circle(2pt);
	\draw[fill ] (-1.6,-2)circle(2pt);
	\draw[fill ] (0.8,-1)circle(2pt);
	\draw[fill ] (-0.8,-1)circle(2pt);
	\draw[fill ] (0,-2)circle(2pt);

	\draw[thin ](0,0)--(1.6,-2);
	\draw[thin ](0,0)--(-1.6,-2);
	\draw[thin ](-1.6,-2)--(1.6,-2);
	
	\node(a)at(-0.4,0.05){$-3$};
	\node(a)at(1.35,-2.3){$-3$};
	\node(a)at(-1.55,-2.3){$-3$};
	\node(a)at(0.4,-1.2){$-2$};
	\node(a)at(-.55,-1.2){$-2$};
	\node(a)at(0,-1.75){$-2$};
	
	\node(a)at(0.35,0.05){$v_1$};
	\node(a)at(1.8,-1.8){$v_3$};
	\node(a)at(-1.9,-1.8){$v_2$};
	\node(a)at(1.15,-.9){$w_2$};
	\node(a)at(-1.2,-.9){$w_3$};
	\node(a)at(0,-2.35){$w_1$};
	\end{tikzpicture}
\caption{}
\label{newex_figure 1}
\end{figure}

Therefore, $(X,0)$ gives an example of a LNE normal surface singularity which was not previously known. 
Indeed, the examples of Lipschitz Normally Embedded surface germs already known are either minimal surface singularities (\cite{NeumannPedersenPichon2020b} or superisolated surface singularities with LNE tangent cone (\cite{MisevPichon2018}).
The cusp singularity  $(X,0) \colon x^5+y^5+z^5+xyz=0$ is not minimal, since it is not rational (its resolution graph is not a tree), nor superisolated, since it is not resolved by the blowup of its maximal ideal.

\begin{proof} 
%
	A generic polar curve $\Pi$ of $(X,0)$ has equation $g(x,y,z)=0$,
	where $g$ is a generic linear combination $\alpha f_x + \beta f_y + \gamma f_z$ of the three partial derivatives $f_x = 5x^4+yz$, $f_y=5y^4+xz$ and $f_z=5z^4+xy$ of $f$. 
	A direct computation in the charts of the point blowups forming $\pi$ shows that the family of generic polar curves has no base points on $X_\pi$, that is $\pi$ factors through the Nash transform of $(X,0)$, and that the $\cal P$-nodes of $(X,0)$ are the three vertices $w_1, w_2$, and $w_3$, with $p_{w_i}=2$.
	Observe that it is easy to see from the defining equation that the projectivized tangent cone of $(X,0)$ consists of three distinct projective lines $C_{v_1}, C_{v_2}$, and $C_{v_3}$, each one corresponding to one of the three $\cal L$-nodes of $(X,0)$.
	The Gauss map $\lambda \colon X_\pi \to \mathrm{Gr}(2,\C^3)$ induces a natural map $\widetilde\lambda \colon \Gamma_\pi \to \mathrm{Gr}(2,\C^3)$ which is constant on each connected component of $\Gamma_\pi \setminus \{\cal P-\text{nodes}\}$, sending any point of the connecting component containing the $\cal L$-node $v_i$ to the projective line $C_{v_i}$.
	Then the fact that $(X,0)$ is LNE is a direct consequence of \cite[Lemma 5.2]{MisevPichon2018} and the Test Curve Criterion \cite[Theorem~3.8]{NeumannPedersenPichon2020a}, repeating the arguments of \cite[\S5]{MisevPichon2018}.
\end{proof}

\bibliographystyle{alpha}                              
\bibliography{biblio}

\vfill

\end{document}